\documentclass[twoside]{article}

\usepackage[accepted]{aistats2021_for_arxiv}

\usepackage[utf8]{inputenc} 
\usepackage[T1]{fontenc}    
\usepackage{hyperref}       
\usepackage{url}            
\usepackage{booktabs}       
\usepackage{amsfonts}       
\usepackage{nicefrac}       
\usepackage{microtype}      

\usepackage{xcolor}
\usepackage{hyperref}

\usepackage{mathtools}

\usepackage{amsmath, amsfonts, amssymb, amsthm}
\usepackage{algorithm,algorithmic}
\usepackage{fancyhdr}
\usepackage{verbatim}
\usepackage{siunitx}
\usepackage{etoolbox}
\usepackage{bm}

\usepackage{ifthen}
\usepackage{xstring}
\usepackage{calc}
\usepackage{bbm,tikz-cd}

\usepackage{setspace}
\usepackage{rotating}
\usepackage{thmtools}
\usepackage{framed}

\usepackage{
nameref,
cleveref,
}

\usepackage{caption}
\usepackage{graphicx}
\usepackage{xcolor}
\usepackage{subfiles}
\usepackage{listings}
\usepackage{subfigure}
\usepackage{lscape}
\usepackage{color}
\usepackage{tabularx}
\newcolumntype{Y}{>{\raggedright\arraybackslash}X}
\usepackage{makecell}
\usepackage{pgfplotstable}
\usepackage{pgfplots}
\pgfplotsset{
	table/search path={plot_figures},
}
\usepackage{tikz}
\usepackage{xr}
\usetikzlibrary{decorations.pathreplacing}
\usepackage{tkz-berge}
\graphicspath{{Figures/}}
\usepackage{subfigure}
\pgfdeclarelayer{bg}    
\pgfsetlayers{bg,main}  
\usepgfplotslibrary{polar}
\pgfplotsset{compat=1.14}

\graphicspath{{Figures/}}

\usepackage{fancybox}

\usepackage{epsfig}
\usepackage{latexsym}
\usepackage{dsfont}
\usepackage{enumerate}
\usepackage[percent]{overpic}
\usepackage[export]{adjustbox}

\makeatletter
\newcommand{\pubdesc}[2]{%
  \csname phantomsection\endcsname 
  \def\@currentlabel{#1}
  \textsc{#1} 
  #2
}
\makeatother

\def \<= {%
  \leqslant}%
\def \>= {%
  \geqslant}%

\newcommand*{\B}[1]{\ifmmode\bm{#1}\else\textbf{#1}\fi}

\newcommand{\s}[3]{\sum \limits_{#1= #2}^{#3}}

\newcommand{\minl}[1]{\min \limits_{ #1 } }
\newcommand{\suml}[1]{\sum \limits_{ #1 } }

\def\la{\langle}
\def\ra{\rangle}

\newcommand{\wt}[1]{\widetilde{ #1 } }

\def\e{\varepsilon}

\def\lm{\lambda}
\def\vp{\varphi}

\def\W{\mathcal W}

\def\Y{\mathcal Y}
\def\Z{\mathcal Z}
\def\M{\mathcal M_+^1}

\newcommand{\cW}{\mathcal{W}}

\def\R{\mathbb R}

\newcommand{\E}{{\mathbb E}}

\newcommand{\wdxi}{\widetilde{\xi}}
\newcommand{\wdlm}{\widetilde{\lambda}}

\newcommand{\wdPsi}{\widetilde{\Psi}}
\newcommand{\wtPsi}{\widetilde{\Psi}}

\def\Blm{\boldsymbol{\lambda}}

\def\BBlm{\boldsymbol{\bar{\lambda}}}

\def\b0{\boldsymbol{0}}
\def\b1{\boldsymbol{1}}

\def\tO{\widetilde{O}}
\def\blm{\bar{\lambda}}

\def\p{\mathtt{p}}
\def\l{\ell}
\def\1#1{%
  \mathbb{#1}}%
\def\2#1{%
  \mathbb{#1}[X]}%
\def\3#1#2{%
\mathbb{#1} / \mathbb{#2}
}
\def\4#1#2{%
\mathbb{#1}[#2]
}
\def\5#1{%
\mathfrak{#1}
}

\def\dfr#1#2{%
\dfrac{#1}{#2}
}
\def\l{%
\left
}
\def\r{%
\right
}
\definecolor{ao(english)}{rgb}{0.0, 0.5, 0.0}

\def \!= {%
  \neq}%
\def \<= {%
  \leq}%
\def \>= {%
  \geq}%

\DeclareMathOperator*{\argmin}{argmin}

\DeclareMathOperator*{\nz}{nnz}
\DeclareMathOperator*{\tgrad}{\widetilde{\nabla}}

\def\eig{\lambda_{\max}(W)}
\def\eigmax{\lambda_{\max}(W)}
\def\eigmin{\lambda_{\min}^+(W)}

\DeclareMathOperator{\nnznum}{nnz}
\def\nnz{\nnznum (\overline{W}_{i})}

\pgfkeys{tikz/mymatrixenv/.style={decoration={brace},every left delimiter/.style={xshift=8pt},every right delimiter/.style={xshift=-8pt}}}
\pgfkeys{tikz/mymatrix/.style={matrix of math nodes,nodes in empty cells,left delimiter={[},right delimiter={]},inner sep=1pt,outer sep=1.5pt,column sep=2pt,row sep=2pt,nodes={minimum width=20pt,minimum height=10pt,anchor=center,inner sep=0pt,outer sep=0pt}}}
\pgfkeys{tikz/mymatrixbrace/.style={decorate,thick}}

\newtheorem{theorem}{Theorem}
\newtheorem{lemma}{Lemma}

\newtheorem{proposition}{Proposition}

\newtheorem{remark}{Remark}

\newtheorem{assumptions}{Assumptions}

\usepackage{wrapfig}
\usepackage{multirow}
\usepackage{enumitem}

\definecolor{darkgreen}{rgb}{0.0, 0.5, 0.0}

\author{%
  David S.~Hippocampus\thanks{Use footnote for providing further information
    about author (webpage, alternative address)---\emph{not} for acknowledging
    funding agencies.} \\
  Department of Computer Science\\
  Cranberry-Lemon University\\
  Pittsburgh, PA 15213 \\
  \texttt{hippo@cs.cranberry-lemon.edu} \\
}

\allowdisplaybreaks

\begin{document}

%

%

\twocolumn[

\aistatstitle{Distributed Optimization with Quantization for Computing Wasserstein Barycenters}

\aistatsauthor{ Roman Krawtschenko \And C\'esar A. Uribe \And  Alexander Gasnikov \And  Pavel Dvurechensky }

\aistatsaddress{ Humboldt University \\ \texttt{r.kravchenko90@gmail.com} \And  MIT \\\texttt{cauribe@mit.edu} \And MIPT,\\ IITP RAS \\ HSE University \\ \texttt{gasnikov@yandex.ru}  \And Weierstrass Institute,  \\ IITP RAS  \\ HSE University \\\texttt{dvureche@wias-berlin.de} }]


\begin{abstract}
We study the problem of the decentralized computation of entropy-regularized semi-discrete Wasserstein barycenters over a network. Building upon recent primal-dual approaches, we propose a sampling gradient quantization scheme that allows efficient communication and computation of approximate barycenters where the factor distributions are stored distributedly on arbitrary networks. The communication and algorithmic complexity of the proposed algorithm are shown, with explicit dependency on the size of the support, the number of distributions, and the desired accuracy. Numerical results validate our algorithmic analysis.

\end{abstract}

\vspace{-0.2cm}
\section{Introduction} \label{S:Intro}
\vspace{-0.2cm}

Optimal transport (OT) has become an important part of modern machine learning for its ability to take into account the geometry of the data for computations.  Applications range from image retrieval \cite{rubner2000earth} and image classification \cite{cuturi2013sinkhorn} to Generative Adversarial Networks \cite{arjovsky2017wasserstein}. An immediate use of Wasserstein distances is the definition of Frechet means of distributions~\cite{agueh2011barycenters,Panaretos2020FrechetMeans}, which is usually called the Wasserstein barycenter (WB)~\cite{Boissard2015,ebert2017construction}. Informally, WB allows us to define for example, an average image when interpreted as a discrete probability distribution. Such flexibility, along side the geometric and statistical properties of WB has led to a large number of applications, e.g., image morphing and image interpolation of natural images~\cite{simon2020barycenters}, averaging atmospheric gas concentration data~\cite{barre2020averaging}, graph representation learning \cite{simou2020node2coords}, fairness in ML \cite{chzhen2020fair}, geometric clustering \cite{mi2020variational}, Bayesian learning~\cite{rios2020wasserstein}, stain normalization and augmentation~\cite{nadeem2020multimarginal}, probability and density forecast combination~\cite{cumings2020probability}, multimedia analysis and fusion~\cite{jin2020multimedia}, unsupervised multilingual alignment~\cite{lian2020unsupervised}, clustering patterns for COVID-19 dynamics~\cite{nielsen2020clustering}, channel pruning~\cite{shen2020cpot}, and many others. 

As OT and, in particular, the WB framework is used in machine learning applications, the scale of the problem increases as well. Thus, the geometric and statistical advantages of WB come with a high computational cost. For example, calculating the WB of two images of one million pixels translates into a large-scale optimization problem, where the definition of the Wasserstein distance itself contains a minimization problem with approximately $10^{12}$ variables. 

Approximating WB with high accuracy requires processing a large number of samples to get good statistical estimates, and dependencies on the distribution support, number of distribution, and desired accuracy are usually high~\cite{borgwardt2019computational}. Therefore, the complexity and scalability of approximating WB is also a main research thrust within the ML community~\cite{fan2020scalable,lin2020computational,puccetti2020computation,lin2020computational}, where some heuristics~\cite{bouchet2020primal},  strongly‑polynomial 2‑approximation~\cite{borgwardt2020lp}, fast computation~\cite{guo2020fast,guminov2019acceleratedAM}, saddle-point \cite{tiapkin2020stochastic}, and proximal methods~\cite{stonyakin2019gradient,xie2020fast} have been proposed, as well as approaches based on multimarginal optimal transport \cite{2019arXiv191000152L,tupitsa2020multimarginal}.

In addition to scale, modern ML applications require other considerations. In particular, decentralized approaches~\cite{NIPS2017_6858} have recently emerged over centralized ones due to its ability to take into account data ownership, privacy, fault tolerance, and scalability, e.g., parameter server~\cite{dean2012large}, federated learning~\cite{konevcny2016federated,mcmahan2016federated,kairouz2019advances}, among others. In this paper, we are focused on: 	$\bullet$ \textit{Communication-efficient}, 	$\bullet$ Stochastic (\textit{semi-discrete}), and 	$\bullet$ \textit{Decentralized} computation of Wasserstein barycenters. \textit{Decentralized} because we assume it is not possible to store the whole dataset into one machine, and probability distributions are stored locally on a number of agents or workers connected over some arbitrary network. The network structure defines some limited communication capabilities between the nodes. \textit{Semi-discrete} because in contrast with most of the available WB computation approaches, we assume the base probability distributions are continuous, while we want to recover their best discrete (finite) barycenter.
Moreover, the semi-discrete setup implies that agents or machines are oblivious to their local probability distribution and can learn about it from the realization of an associated random variable. Finally, \textit{Communication-efficient} because we exploit the fact that the gradients of the associated dual formulation~\cite{Scaman2017Optimal,uribe2020dual} of the WB lie in the probability simplex, i.e., is a discrete distribution. Thus, we propose a quantized communication approach that drastically reduces the network's communication load by sampling from such a discrete distribution, creating a sparse approximation of the stochastic gradient. Applied to the WB problem, this means that instead of sending a whole barycenter estimate (e.g., an image, in each iteration), agents share a histogram from a finite number of samples drawn from its current WB estimate. 

The main \textbf{contributions} of this paper are:

\vspace{-0.5cm}
\begin{itemize}[leftmargin=*]
\setlength\itemsep{-0.4em}
    \item We propose a general quantized communication-based algorithm for the distributed computation of semi-discrete WB over networks.
    \item We analyze the communication and communication complexity of the proposed algorithm in two setups: 1) The number of samples from the distribution and the samples from the gradient are allowed to increase with time. 2) The number of samples is fixed to some predefined constant. 
    \item Our specific results for the computation of WB are based on a new general primal-dual analysis of the accelerated stochastic gradient descent method, for which we provide a new convergence rate and sample complexity analysis.
    \item We provide numerical experiments that empirically support our theoretical finding on the convergence of the proposed algorithm.
\end{itemize}

\vspace{-0.5cm}

\textbf{Related works:} 
Efficient gradient quantization was studied in~\cite{horvth2019gradientquant,koloskova2019decentralized,koloskova2019decentralized2}. However, only non-accelerated stochastic gradient approaches are available. A similar approach to ours was recently studied in~\cite{ballu2020stochastic}, which uses SGD and sampling scheme in the discrete setting. In contrast, we use AGM (Accelerated gradient method) instead of SGD (Stochastic gradient descent), and we study the WB problem in the semi-discrete setting and give precise computation and communication complexity in terms of the dimension $n$ of the barycenter space, the number of considered distributions $m$, and the network architecture. Centralized Semi-discrete WB~\cite{Claici2018StochasticWB} was studied in \cite{NIPS2017_6858}, where the convergence of the discrete approximations of the continuous measures was studied. Decentralized discrete WB communication and computation complexity was studied in~\cite{uribe2018distributed}, and their corresponding semi-discrete analysis is available in~\cite{dvurechensky2018decentralize}, both approaches with no quantization for communication efficiency. Empirical studies of the approximation of the fully continuous barycenter are recently available in~\cite{li2020continuous}. To the best of the authors' knowledge, the best estimate of the complexity of computing the WB in the centralized discrete setup is available in~\cite{2010.04677}.

This paper is organized as follows. Section~\ref{sec:Problem} introduces the problem of distributed WB computation. Section~\ref{sec:general} presents the general primal-dual accelerated gradient method. Section~\ref{sec:dual} analyzes the primal-dual properties of the WB problem. Section~\ref{sec:main} states our proposed algorithm and main results. Section~\ref{sec:numerics} shows the numerical results. Finally, Section~\ref{sec:Discussion} provides conclusions and future work.

\textbf{Notations:}
We denote $\mathcal{M}^1_+(\mathcal{X})$ as the set of positive Radon probability measures on a metric space $\mathcal{X}$.
$Y\sim \mu$ means that random variable $Y$ is distributed according to measure $\mu$. Let $e_{i}$ denote the $i$-th unit vector. Let $W$ be a positive semi-definite matrix with maximal and minimal nonzero eigenvalues denoted by $\eigmax$ and $\eigmin$. We denote the condition number of a matrix $W$ by $\chi(W)\triangleq {\eig}/{\eigmin}$. $ S_1(n) \triangleq \{ p \in \R^n : \sum_{ l= 1}^n p_l = 1\}$ denotes the probability simplex of dimension $n$. $\tilde{O} (\cdot )$ denotes the complexity up to polylogarithms.  $\delta$ denotes the Dirac mass and $[p]_i$ denote the $i$-th component of the vector $p$. $\otimes$ denotes the Kronecker product of matrices and $\odot$ denotes the component-wise vector multiplication.
$ \nz(x)$  denotes the number of non-zero elements of vector $x$. Given a matrix $\overline{W}$ and its rows $\overline{W}_i$ we write $\kappa (\overline{W}) \triangleq  \sum_{i=1}^m \nnz $.


\vspace{-0.2cm}
\section{Problem Statement and Results}\label{sec:Problem}
\vspace{-0.2cm}

In this section, we describe the problem of the distributed computation of the entropy-regularized semi-discrete Wasserstein barycenter. Moreover, we present a summary of the main results of this paper regarding the communication, sample, and arithmetic complexity of the proposed algorithm.

\begin{figure}[!]
    \centering
    \includegraphics[width=0.7\linewidth]{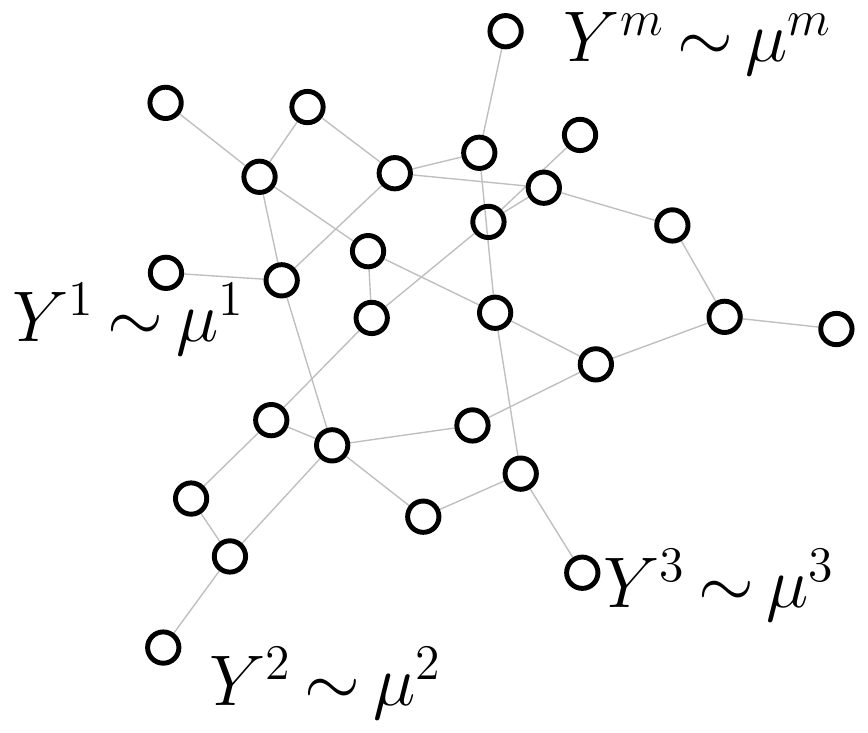}
    \caption{A network of agents, where each agent is able to sample from a different distribution, i.e., agent $i\in V$ samples from $\mu^i$. }
    \label{fig:network}
\end{figure}

Consider a network of $m$ agents defined over a connected, fixed, and undirected graph \mbox{$\mathcal{G} = (V,E)$}, where $V = \{1,\cdots,m\}$ is the set of agents and $E$ is the set of edges, such that $(i,j)\in E$ if agent $i$ and agent $j$ can communicate with each other. We assume that the graph $\mathcal{G}$ does not have self-loops. Moreover, at each time $k\geq 0$, each agent has access to \textit{realizations} or \textit{samples} (the number of samples will be specified later) from a stationary random process $\{Y^i_k\}_{k\geq 0}$ such that $Y^i_k \sim \mu^i$, where $\mu^i \in \mathcal{M}_{+}^1(\mathcal{Y})$ is a continuous probability distribution with with density $q^i(y)$ on a metric space $\mathcal{Y}$, see Figure~\ref{fig:network}.

The group of agents tries to collaboratively compute the entropy-regularized semi-discrete Wasserstein barycenter, defined as the discrete probability distribution $\nu$, on a fixed support $z_1, \dots, z_n \in \mathcal{Z}$, i.e., $\nu=\sum_{i=1}^n[p^*]_i\delta(z_i)$, and $p \in S_1(n)$ such that

\vspace{-0.6cm}

\begin{align}\label{eq:PrimPr}
p^* & = \argmin_{p \in S_1(n)} \sum\limits_{i=1}^{m} \omega_i \W_{\gamma,\mu_i}(p), 
\end{align}

\vspace{-0.2cm}

where $\{\omega_i\}_{i=1}^m$ is a set of non-negative weights with $\sum_{i=1}^m\omega_i=1$. Without loss of generality we set $\omega_i = 1/m$. Moreover, we denote $\W_{\gamma,\mu}(p) \triangleq \W_{\gamma}(\mu,p)$ as the fixed probability measure $\mu$, where $\W_{\gamma}(\mu,p)$ is the entropy-regularized semi-discrete Wasserstein distance:

\vspace{-0.7cm}

\begin{align}\label{WassDis}
\W_{\gamma}(\mu,p)=\min_{\pi \in \Pi(\mu,p)}\Bigg\{ \sum_{i=1}^n\int_{\Y} c_i(y)\pi_i(y)dy + \nonumber  \\
\gamma \sum_{i=1}\limits^n\int_{\Y}
\pi_i(y)\log\left(\frac{\pi_i(y)}{\xi}\right)dy\Bigg\},
\end{align}

\vspace{-0.4cm}

where $\xi$ is the density of uniform distribution on $\Y \times \mathcal{Z}$, $c_i(y) = c(z_i,y)$ denotes the cost of transporting a unit of mass from point $z_i \in \mathcal{Z}$  to point $y \in \mathcal{Y}$, and $\Pi(\mu,p)$ is the set of admissible coupling measures $\pi$ defined as

\vspace{-0.6cm}

\begin{align*}
  \Pi(\mu,p) = \Bigg\{\pi \in\M(\Y) \times S_1(n) &:  \\
  \sum_{i=1}^n \pi_i(y) = q(y), y \in
\mathcal{Y}  &, \int_{\mathcal{Y}} \pi_i(y)dy = p_i  \Bigg\}
\end{align*}

\vspace{-0.2cm}

and $\gamma$ is the regularization parameter.

Note that unlike \cite{genevay2016stochastic}, we regularize the problem by the Kullback-Leibler divergence from the uniform distribution $\xi$, which allows us to find explicitly the Fenchel conjugate for $\W_{\gamma}(\mu,p)$, to be defined later in Lemma~\ref{lm:variance_stoch_gradient}.

We are interested in the distributed computation of approximate solutions of~\eqref{eq:PrimPr}, such that, each agent finds an approximate distribution $\tilde p_i$, where
\begin{equation}   \label{eq:Precise Wasserstein Barycenter Approximation}
  \frac{1}{m}\suml{i=1}^{m}  \cW_{\gamma} ( \mu_{i} , \tilde p_i ) -
 \frac{1}{m}\suml{i=1}^{m}  \cW_{\gamma}  ( \mu_{i} , p^{*} ) \<= \e,
\end{equation}
and $ \tilde p_1,\cdots,\tilde p_m$ are close to each other in a sense to be defined soon. 
Also, note that with an appropriate selection of the regularization parameter $\gamma$, one can also compute an approximate solution to the non-regularized Wasserstein barycenter as described in the next proposition.

\begin{proposition}\label{prop:volume}
  Assume that the space $\Y \times \mathcal{Z}$ is compact with the volume $\Omega \triangleq {\rm Vol}(\Y \times \mathcal{Z})$ and there exist $\tilde p$ satisfying \eqref{eq:Precise Wasserstein Barycenter Approximation} with $\gamma = \hat{\gamma}\triangleq {\e}/{(4 \log(\Omega) )}$.
  Then
\begin{align}
  \label{eq:16}
  \frac{1}{m}\suml{i=1}^{m}  \cW_0 ( \mu_{i} , \tilde p) -
 \frac{1}{m} \suml{i=1}^{m}  \cW_0 ( \mu_{i} , p^{*} ) \<= 2\e.
\end{align}
\end{proposition}

\begin{remark}
Formally, the $\rho$-Wasserstein distance for $\rho \geq 1$ is $\left({\W_0(\mu,\nu)}\right)^{\frac{1}{\rho}}$ if $\Y=\Z$ and $c_i(y) = d^{\rho}(z_i,y)$, $d$ being a distance on $\Y$.
For simplicity, we refer to~\eqref{WassDis} as regularized Wasserstein distance. However, we want to emphasize that our algorithm does not rely on any specific choice of cost $c_i(y)$ and can be used for the Monge-Kantorovich distances where the transportation costs do not have to satisfy the metric properties.
\end{remark}

The objective of the group of agents is to solve~\eqref{eq:PrimPr}. However, the distributed structure induced by the network $\mathcal{G}$ imposes communication constraints we need to consider. We assumed that at each time $t\geq 0$, each agent $i\in V$ has access to the samples from the distribution $\mu^i$ only, i.e., agent $i$ has no information about $\mu^j$ for $j\neq i$. Therefore, cooperation is needed. Such cooperation is modeled as the ability of the agents to share information over the set of edges $E$.


 \begin{table*}[t!]
\label{results table}
  \centering
   \caption{A summary of complexity bounds for \mbox{\textbf{Q-DecPDSAG}}. Constant and logarithmic terms are hidden for notation convenience.  }

  \resizebox{0.8\linewidth}{!}{
    \begin{tabular}{cccc}
    \toprule
    \multirow{2}{6em}{\textbf{Complexity}}       & \textbf{DecPDSAG} &  \multicolumn{2}{ c }{\textbf{Q-DecPDSAG}}\\
             & \cite{dvurechensky2018decentralize}  & \textbf{Increasing Batch} & 
                \textbf{Constant Batch}
                \\
                \midrule
\makecell{ Communication \\ Rounds}
      & $\dfr{\sqrt{\chi (W) n} }{\e} $ & $\dfr{\sqrt{\chi (W) n}}{\e}  $  & $\dfr{ \chi (W) n}{\e^2}  $
      \\  \midrule

\makecell{Total Arithmetic \\ Operations  }
&$ \chi (W) \max \{ \dfr{m n^{2}}{\e^{2}} ,
\dfr{n\kappa (W) }{\e^{2}}\} $     &
$\chi (W) \max \{ \dfr{m n^{2}}{\e^{2}} ,
\dfr{n \kappa (W)}{\e^{2}}\}$
&
 $ \dfr{\chi (W) n^{2} \kappa (W) }{\e^{2}}   $
\\
\midrule
      \makecell{Total Number of \\ bits sent}
      &
      $    \dfrac{\sqrt{\chi (W) } m n^{3/2}}{ \e}  $
          &  $ \sqrt{\chi (W) } 
\kappa (W)  \max \{
     \dfrac{ n^{3/2} }{\e},  \dfr{n   }{\e^{2}} \}
$
          &
           $                                                                                                                \dfr{\chi (W)   n  \kappa (W) }{\e^{2}}$
            \\  \midrule
\makecell{Total Arithmetic \\ Operations  \\ expander graph}
&$  \dfrac{  m n^2 }{\e^{2} }$     &    $ \dfrac{  m n^2 }{\e^{2} } $
&
 $\dfr{m n^{2}}{\e^2}  $
\\
\midrule
      \makecell{Total Number of \\ bits sent \\  expander graph}
      &
      $    \dfrac{ m n^{3/2}}{ \e}  $
          &  $
     \max \{ \dfrac{m n^{3/2} }{\e} , \dfr{m n   }{\e^{2}} \}
            $
          &
           $\dfr{ m n}{\e^2}  $
      
      \\
      \midrule
      \makecell{Number of bits sent \\at iteration $k$ }
&$ n$          &    $\max\{n,k\} $    &
$ M $
               \\ \midrule
      \makecell{Total Number \\ of Samples}
      &  $ \dfrac{\sqrt{ \chi (W)}  mn }{    \e^{2}}  $ & $ \dfrac{\sqrt{ \chi (W)}  mn }{    \e^{2}}  $ &  $ \dfrac{ \chi (W) m n }{    \e^{2}}  $
      \\
        \bottomrule
    \end{tabular}
}
  \label{tab:results_table}
  \end{table*}

In Table~\ref{tab:results_table}, we informally summarize the main results of this paper. We also describe the communication and the computational complexity of the proposed algorithm, quantized distributed primal-dual stochastic accelerated gradient method \mbox{\textbf{Q-DecPDSAG}}. We compare the obtained bounds with the Non-quantized version denoted \textbf{DecPDSAG}~\cite{dvurechensky2018decentralize}. The specifics of the proposed quantization scheme will be described in Section \ref{sec:quantized_gradients}. Moreover, we compare two different regimes for \mbox{\textbf{Q-DecPDSAG}}, namely, having an increasing number of samples per iteration and having a constant number of samples per interaction. We analyze the computational complexity of \mbox{\textbf{Q-DecPDSAG}} in terms of arithmetic operations and the total communication complexity, as the total number of coordinates sent per node over all the iterations of the algorithm for the general graph as well as for the specific architecture where we use expander graphs.

In the next section, we study the primal-dual properties of the WB problem. We specifically exploit the dual problem structure and propose an effective sampling strategy for the construction of estimates of the stochastic gradient. This will be the main building block of our algorithm.

\vspace{-0.2cm}
\section{Duality and Quantization }\label{sec:dual}
\vspace{-0.2cm}

In this Section, we build upon the dual properties of the entropy-regularized Wasserstein distance to develop a distributed algorithm that allows for the computation of a Wasserstein barycenter.

\vspace{-0.2cm}
\subsection{Duality of the WB problem}
\vspace{-0.2cm}

A commonly used method to introduce the distributed structure of the problem is by reformulating~\eqref{eq:PrimPr} to take into account the constraints imposed by the graph $\mathcal{G}$ explicitly. To do so, we define the Laplacian matrix $\bar W{\in \mathbb{R}^{m\times m}}$ of the graph $\mathcal{G}$ as $[\bar W]_{ij} \triangleq -1$, $(i,j) \in E$, $[\bar W]_{ij} =\text{deg}(i)$ if $i= j$, and $0$ otherwise, where $\deg(i)$ is the number of neighbors of node $i$, i.e., the degree of the node $i$. Similarly, we define the auxiliary matrix $W \triangleq \bar W \otimes I_n$, which takes into account the dimension $n$ of the decision variable. Therefore, assuming that the graph $\mathcal{G}$ is undirected and connected, $\bar W$ is symmetric and positive semidefinite. Moreover, the matrix $W$ inherits the spectral properties of $\bar W$. More importantly, the Laplacian matrix has the property: $ W{\mathtt{p}} = 0  \;  \; \text{if and only if} \; \;  {p_1 = \cdots = p_m}$,
where $\mathtt{p} = [p_1^T,\cdots,p_m^T]^T \in \mathbb{R}^{mn}$.
The relation $\sqrt{W}{\mathtt{p}} = 0  \; \;\text{if and only if}  \; \; {p_1 = \cdots = p_m}$ also holds~\cite{dvurechensky2018decentralize}. Thus, we equivalently rewrite problem~\eqref{eq:PrimPr} as

\vspace{-0.8cm}

    \begin{align}\label{eq:DualiazationProcces}
    \max_{\substack{p_1,\dots, p_m \in S_1(n) \\ \sqrt{W} \mathtt{p}=0 }} ~ - \frac{1}{m}\sum\limits_{i=1}^{m}  \W_{\gamma, \mu_i}(p_i) .
    \end{align}
    
   \vspace{-0.2cm}

Given that~\eqref{eq:DualiazationProcces} is an optimization problem with linear constraints, we introduce a vector of dual variables $\Blm = [\lambda_1^T,\cdots,\lambda_m^T]^T \in \R^{mn}$ for the constraints $\sqrt{W}\p=0$. Then, the Lagrangian dual problem for \eqref{eq:PrimPr} can be written as 

\vspace{-0.8cm}

\begin{align}\label{eq:DualProblem}
    \min_{\Blm
\in \R^{mn}} \W_{\gamma}^*(\Blm) \triangleq \frac{1}{m}\sum_{i=1}^{m} \W^*_{\gamma, \mu_i}(m[\sqrt{W}\Blm]_i),
\end{align}

\vspace{-0.2cm}

where $\W^*_{\gamma,\mu_i}(\cdot)$ is the Fenchel-Legendre transform of $\W_{\gamma,\mu_i}(p_i)$ defined as

\vspace{-0.5cm}

{\small 
\begin{align*}
    \W^*_{\gamma, \mu_i}(m[\sqrt{W}\Blm]_i) & {=} \max_{p_i\in S_1(n)} \big\lbrace \langle \lambda_i, [\sqrt{W}\mathtt{p}]_i\rangle-\frac{1}{m}\W_{\gamma, \mu_i}(p_i)\big\rbrace,
\end{align*}
}

\vspace{-0.5cm}

and $[\sqrt{W}\p]_i$ and $[\sqrt{W}\Blm]_i$ denote the $i$-th $n$-dimensional block of vectors $\sqrt{W}\p$ and $\sqrt{W}\Blm$ respectively. 

The next two auxiliary lemmas state properties of $\W^*_{\gamma, \mu_i}(\cdot)$ what will be useful for our analysis~~\cite{dvurechensky2018decentralize}. In particular,  $\W^*_{\gamma, \mu_i}(\cdot)$ is a smooth function with Lipschitz-continuous gradient and can be expressed as an expectation of a  function of additional random argument.

\begin{lemma}[Lemma $1$ in \cite{dvurechensky2018decentralize}] \label{lm:Lipschitz_gradient} 
Given a positive Radon probability measure  $\mu \in \mathcal{M}_{+}^1(\mathcal{Y})$  with density $q(y)$  on a metric space $\mathcal{Y}$, the Fenchel-Legendre dual function of $\W_{\gamma,\mu}(p)$ can be written as

\vspace{-0.6cm}

\begin{align*}
\W_{\gamma,\mu}^*(\blm) {=}\E_{Y\sim\mu}\Bigg[ \gamma
\log\Bigg(\frac{1}{q(Y)}\sum_{\ell=1}^n\exp\Bigg(\frac{[\blm]_\ell{-}c_\ell(Y)}{\gamma}
\Bigg)\Bigg)\Bigg].
\end{align*}

\vspace{-0.2cm}

Moreover, $\W_{\gamma,\mu}^*(\blm)$ has $m/\gamma$-Lipschitz gradients w.r.t. $2$-norm, and its $l$-th coordiante, for $l=1,\dots,n$,  is defined as
 \begin{align*}
[\nabla \W_{\gamma,\mu}^*(\bar{\lambda})]_l & = \E_{Y\sim\mu}
 \Bigg[ \frac{\exp(([\bar{\lambda}]_l-c_l(Y))/\gamma)
}{\sum_{\ell=1}^n\exp(([\bar{\lambda}]_\ell-c_\ell(Y))/\gamma)}\Bigg].
\end{align*}
\end{lemma}

\begin{lemma}[Lemma $2$ in \cite{dvurechensky2018decentralize}]\label{Lm:dual_obj_properties2}
  The function $\W_{\gamma}^*(\Blm)$ in~\eqref{eq:DualProblem} has $ m \lambda_{\max}(W) /
  \gamma$-Lipschitz gradients w.r.t. $2$-norm. Moreover,  for $l=1,\dots,n$, it holds that
  \begin{align}\label{eq:DualGrad}
\left[\nabla \W^*_\gamma (\Blm) \right]_l & = \sum_{j=1}^{m}\sqrt{W}_{lj} \nabla \W_{\gamma,\mu_j}^*(\blm_j),
\end{align}
where we denoted $\blm_j = m[\sqrt{W}\Blm]_j$.
\end{lemma}

\vspace{-0.2cm}


Lemma~\ref{Lm:dual_obj_properties2} states that the \emph{dual} problem~\eqref{eq:DualProblem} is a smooth stochastic convex optimization problem. This is in contrast to~\cite{lan2017communication}, where the primal problem is a stochastic optimization problem. This will be the main observation that will allow us to propose an efficient sampling and communication strategy to find an approximate solution of \eqref{eq:DualProblem}, and~\eqref{eq:PrimPr} respectively. More importantly, the gradient of the dual function can be computed in a distributed manner over a network, where each agent $j \in V$ computes $\nabla \W_{\gamma,\mu_j}^*(\blm_j)$ using local information only. Then, the gradient is shared with neighbors following the graph topology, and a full gradient step can be taken.

Next, we describe our sampling and quantization approach. We take advantage of the smooth stochastic form of the dual problem, for which each agent can obtain an estimate of the gradient-based on samples from the random variables $Y^i$. Moreover, such an approximate gradient will be a probability distribution itself. Thus, instead of communicating the full gradient, an agent can share a histogram of samples of the approximate gradient itself.

\vspace{-0.2cm}
\subsection{Quantized stochastic gradients}\label{sec:quantized_gradients}
\vspace{-0.2cm}

We build our main result based on the idea of double sampling of the gradients. In~\cite{dvurechensky2018decentralize}, the authors propose to use sampling from measures $\mu_i$ to approximate the gradients of the dual problem. Given that each agent $i$ can obtain at each iteration $M_1$ samples of the random variable $Y_k^i \sim \mu_i$, one can define an approximate gradient as 
\begin{align}\label{eq:approx_grad}
    \widehat{\nabla} \W_{\gamma,\mu_j}^*(\blm_j) & = \frac{1}{M_1}\sum_{r=1}^{M_1} p_j(\blm_j , Y_r^j),
\end{align}
where, for all $l=1,\cdots,n$,
\begin{align}\label{eq:pes}
    [p_j(\blm_j, Y_r^j)]_l & {=}
\frac{\exp(([\bar{\lambda}_j]_l-c_l(Y_r^j))/\gamma)
}{\sum_{k=1}^n\exp(([\bar{\lambda}_j]_{k}-c_k(Y_r^j))/\gamma)}.
\end{align}

\vspace{-0.2cm}

However, such an approximate gradient will be a vector of dimension $n$. For the case where $n$ is large, this might be prohibitively expensive in terms of communications. Thus, we take advantage of the fact that $\widehat{\nabla} \W_{\gamma,\mu}^*(\bar{\lambda}) \in S_1(n)$, to propose a strategy that guarantees lower communication costs. We propose to use the approximate stochastic gradient defined in the next lemma.

\begin{lemma}\label{lm:variance_stoch_gradient}
Let each agent $i \in V$, at iteration $k$, take $M_{i,1}$ samples of the random variable $Y^i$, and build a stochastic gradient as defined in~\eqref{eq:approx_grad}. Now, let each agent $i \in V$ take $M_{i,2}$ samples of a discrete random variable $\xi$ on $l \in \{1,...,n\} $ with $\xi = l$ with probability $[\widehat{\nabla} \W_{\gamma,\mu_j}^*(\blm_j)]_l$, and construct a histogram, or approximate quantized sampled gradient of its local dual function 
as $ 
       \widetilde{\nabla} \W_{\gamma,\mu_j}^*(\blm_j) = \frac{1}{M_{i,2}}\sum_{r=1}^{M_{i,2}}
e_{\xi_r}$, which are then combined in the stochastic approximation for the whole dual objective \eqref{eq:DualProblem} defined as
\[
 \left[\widetilde{\nabla} \W^*_\gamma (\Blm) \right]_l  = \sum_{j=1}^{m}\sqrt{W}_{lj} \widetilde{\nabla} \W_{\gamma,\mu_j}^*(\blm_j).
\]
 Then, the quantized stochastic gradient is unbiased, i.e.,
 
 \vspace{-0.4cm}
 
\begin{equation}\label{eq:expected value}
 \quad \E \widetilde{\nabla} \W_{\gamma}^*(\Blm)  = \nabla \W_{\gamma}^*(\Blm) .
\end{equation}

\vspace{-0.2cm}

Furthermore, its variance 
is bounded by

\vspace{-0.8cm}

{\small
\begin{align*}
&\E \|\widetilde{\nabla} \W_{\gamma}^*(\Blm)  {-} \nabla \W_{\gamma}^*(\Blm)\|_2^2 \nonumber \leq
2  \lambda_{\max}(W)   \sum_{i=1}^m \Big( \dfr{1}{M_{i,1}} {+} \dfr{1}{M_{i,2}} \Big).
\end{align*}
}
\end{lemma}

\vspace{-0.4cm}

There are two main advantages in the approach described in Lemma~\ref{lm:variance_stoch_gradient}. First, the size of the communicated messages between the nodes diminishes with a smaller second batch, while the double-batched algorithm preserves roughly the same complexity just as the Accelerated Distributed Computation of Wasserstein barycenter in~\cite{dvurechensky2018decentralize}. The second advantage is the quantization aspect such that agents exchange sparse integer vectors instead of dense vectors of float variables.

The main observation of this Section is that the dual problem to the WB setup is a stochastic optimization problem. Moreover, its stochastic gradients live in a probability simplex, which defines a discrete probability distribution. Such a unique structure allows for effective sampling strategies, which translate into effective communication due to the primal-dual properties of the problem. In the next section, we propose a new analysis of stochastic optimization problems with such a structure. Later we will specialize it for the WB problem, but our general results could be of independent interest.

  \vspace{-0.2cm}
\section{General Primal-Dual Accelerated Stochastic Gradient Method}\label{sec:general}
\vspace{-0.2cm}

For any finite-dimensional real vector space $E$, we denote by $E^*$ its dual, by $\la \lambda, x \ra$ the value of a linear function $\lambda \in E^*$ at $x\in E$. Let $\|\cdot\|$ denote some norm on $E$ and $\|\cdot\|_{*}$ denote the norm on $E^*$ which is dual to $\|\cdot\|$, i.e.
$\|\lambda\|_{*} = \max \{ \la \lambda, x \ra :  \|x\| \leq 1 \}$.
For a linear operator $A:E_1 \to E_2$, we define the adjoint operator $A^T: E_2^* \to E_1^*$ in the following way $\la u, A x \ra = \la A^T u, x \ra, \quad \forall ~ u \in E_2^*, x \in E_1$.
We say that a function $f: E \to \R$ has a $L$-Lipschitz continuous gradient w.r.t. norm $\|\cdot\|_{*}$ if it is continuously differentiable and its gradient satisfies Lipschitz condition $
\|\nabla f(x) - \nabla f(y) \|_{*} \leq L \|x-y\|, \quad \forall ~x,y \in E$.

 \vspace{-0.2cm}

The main problem we consider in this section is 
\begin{equation}
\label{eq:primal_problem}
\min_{x \in Q} \{f(x) \;\;:\;\; Ax=b\},
\end{equation}
where $Q$ is a simple closed convex set, $A: E \to H$ is given linear operator, $b \in H$ is given, $\Lambda = H^*$. 
The dual problem for \eqref{eq:primal_problem} consists in minimization of the function

\vspace{-0.9cm}

\begin{align}\label{eq:dual_problem}
\vp(\lambda) & \triangleq\la \lambda, b \ra +  \max_{x\in Q}\{ -f(x) - \la A^T \lambda  ,x \ra \} \nonumber \\
&= \la \lambda, b \ra + f^*(-A^T\lambda). 
\end{align}

\vspace{-0.2cm}

Here $f^*$ is the Fenchel-Legendre conjugate function for $f$.  
We say that a (random) point $\hat{x}$ is an $\e-$solution to Problem~\eqref{eq:primal_problem} if
\begin{align}\label{eq:eps_sol_def}
f(\E\hat{x})-f^* \leq \varepsilon, \;\; \|A\E\hat{x}-b\|_2\leq \frac{\varepsilon}{R},
\end{align}
where $R$ is such that the dual problem~\eqref{eq:dual_problem} has a solution $\lambda^*$ satisfying $\|\lambda^*\|_{2} \leq R < +\infty$.

We assume that $f^*(-A^T\lambda)= \E_\xi F^*(-A^T\lambda,\xi)$, where $\xi$ is random vector and $F^*$ is the Fenchel-Legendre conjugate function to some function $F(x,\xi)$, i.e. it satisfies $F^*(-A^T\lambda,\xi) = \max\limits_{x\in Q}\{\la -A^T \lambda,x\ra - F(x,\xi) \}$. $F^*(\blm,\xi)$ is assumed to be smooth and, hence $\nabla_{\blm} F^*(\blm,\xi) = x(\blm,\xi)$, where $x(\blm,\xi)$ is the solution of the maximization problem

\vspace{-0.8cm}

\begin{align*}
x(\blm,\xi) = \arg\max\limits_{x\in Q}\{\la \blm,x\ra - F(x,\xi)\}.
\end{align*}

\vspace{-0.2cm}

Further, we assume that the dual problem~\eqref{eq:dual_problem} can be accessed by a stochastic oracle $(\Phi(\lambda,\xi),~ \nabla \Phi(\lambda,\xi))$ with $\Phi(\lambda,\xi) = \la \lm, b\ra +F^*(-A^T\lambda,\xi)$ and $\nabla \Phi(\lambda,\xi) = b - A \nabla F^*(-A^T\lambda,\xi)$. Let us summarize our main assumptions as follows.
\begin{assumptions}
\label{main assumptions}
\begin{itemize}
  \item
    The gradient of the objective function $\vp$ is $L$-Lipschitz continuous.
\item The stochastic oracle pair $(\Phi(\lambda,\xi) , \nabla \Phi(\lambda,\xi) )$ for the dual problem satisfies
\begin{align}
  \E_\xi \Phi(\lambda,\xi) &= \vp(\lambda)  \\
   \E_\xi  \nabla \Phi(\lambda,\xi) &= \nabla \vp(\lambda) .
\label{eq:stochastic oracle pair}
\end{align}
\item There is a random variable $\wdxi$, which is independent of $\xi$, and a stochastic approximation $\tgrad \Phi (\lambda, \xi, \wdxi ) $ for $\nabla \Phi(\lambda,\xi)$ satisfying
\begin{align}
  &\E_{\wdxi} \tgrad \Phi (\lambda, \xi, \wdxi ) = \nabla \Phi(\lambda,\xi), \;\; \forall \lambda, \xi  \\
   &\E_{\xi,\wdxi} \|\tgrad \Phi (\lambda, \xi, \wdxi )  - \nabla \vp(\lambda) \|_2^2  \leq \sigma^2, \;\; \forall \lambda.
\label{eq:stochastic oracle pair 2}
\end{align}
\item The dual problem~\eqref{eq:dual_problem} has a solution $\lambda^*$ and there exists some $R >0$ such that $\|\lambda^*\|_{2} \leq R < +\infty$.
\end{itemize}
\end{assumptions}


We propose Algorithm~\ref{Alg:APDSGD} for solving~\eqref{eq:primal_problem} using the stochastic oracle described in Assumption~\ref{main assumptions}. Contrary to~\cite{dvurechensky2018decentralize}, we extend the accelerated stochastic gradient method from~\cite{devolder2011stochastic} and propose its primal-dual version. The main benefit is flexibility in the choice of the batch size, which is assumed fixed in~\cite{dvurechensky2018decentralize}. Also, the algorithm from~\cite{devolder2011stochastic} cannot be directly used since it does not allow us to reconstruct the solution of the primal problem, i.e., the barycenter in our application. We believe that this general primal-dual accelerated stochastic gradient method with flexible batch size can be of independent interest. The next theorem presents the convergence properties of Algorithm~\ref{Alg:APDSGD}.

\begin{theorem}
    \label{th:conv general alg}
Let Algorithm~\ref{Alg:APDSGD} be applied to solve Problem~\eqref{eq:primal_problem} under the above assumptions. 
\textbf{Case A} 
Assume also that at each iteration the variance $\sigma_k$ of the stochastic gradient satisfies $\sigma_k^2 \<= {\e L \alpha_k}/{A_k}$. 
  Then, Algorithm~\ref{Alg:APDSGD} outputs an $\e-$ solution to Problem~\eqref{eq:primal_problem}
  after $\sqrt{{8LR^{2}}/{\e}} $ iterations. \textbf{Case B} 
  Assume the variance $\sigma^2$ is constant. Then, Algorithm~\ref{Alg:APDSGD} outputs an $\e-$solution to Problem~\eqref{eq:primal_problem} after $ \max \{\sqrt{{4 LR^{2}}/{\e}}, {9 \sigma^{2} R^{2}}/{\e^{2}} \}$ iterations.  
\end{theorem}

\begin{algorithm}[!t]
\caption{Primal-Dual Accelerated Stochastic Gradient Method}
\label{Alg:APDSGD}
{\small
\begin{algorithmic}[1]
    \STATE  \textbf{Initialization}\\
  Set $\eta_{0} = \zeta_{0} = z_{0}  =  \lambda_{0} = 0$
  and choose a number of iterations $N$.
    \FOR{$k=0,\dots, N-1$}
        \STATE Select the coefficients: 
        
         \textbf{Case A) } \\
    $A_{k }= \dfr{1}{2}  (k+1)(k+2)$, $\alpha_k =  \dfr{k+1}{2}$ with
    $\beta_{k}=  2L$.
    
         \textbf{Case B) }\\
        $\alpha_{k}{=} \dfr{k{+}1}{2\sqrt{2}}$, $\beta_{k}{=} L {+}\dfrac{\sigma (k+2)^{3/2}}{2^{1/4}\sqrt{3} R}$,
        $A_{k}{=} \dfr{(k{+}1)(k{+}2)}{4 \sqrt{2}}$. 
                    
        \STATE $ \tau_{k} = {\alpha_{k+1}}/{A_{k+1}}$.
        \STATE $ z_{k} = - \dfrac{1}{\beta_{k}}   \suml{l=0}^{k} \alpha_{l} \tgrad \Phi (\lambda_{l}, \xi_{l}, \wdxi_{l} )$ 
        \STATE $\lambda_{k+1} = \tau_{k}z_{k} + (1 - \tau_{k} )  \eta_{k}$.
        \STATE $ \zeta_{k+1} =  z_{k} -  \dfrac{\alpha_{k+1}}{\beta_{k}}  \tgrad \Phi (\lambda_{k+1}, \xi_{k+1}, \wdxi_{k+1} )$.
        \STATE $\eta_{k+1} = \tau_{k} \zeta_{k+1} + (1- \tau_{k}) \eta_{k} $.
    \ENDFOR
    \STATE $\hat{x}_{k+1} = \frac{1}{A_{k+1}}\sum_{l=0}^{k+1} \alpha_l 
                    x(-A^T\lambda_{k+1},\xi_{k+1}) $.
    \ENSURE The point $\hat{x}_{N}$. 
\end{algorithmic}}
\end{algorithm}

\vspace{-0.3cm}
\section{Proposed Algorithm and Analysis}\label{sec:main}
\vspace{-0.2cm}

In this section, we use the main primal-dual algorithm presented in Section~\ref{sec:general} to the specifics of Wasserstein Barycenter problem of Section~\ref{sec:Problem}. The algorithm has two different options for the batch size, one for the increasing batch size similar to the main algorithm in \cite{dvurechensky2018decentralize} and one for the constant batch size. We adapt the structure of the general Algorithm~\ref{Alg:APDSGD} for the specific properties of the WB. This allows us to propose Algorithm~\ref{Alg:mainWB} for the distributed computation of WB in the semi-discrete setting with quantized communications.  Finally, the following theorems establish the complexity results in Table \ref{tab:results_table}.

\begin{algorithm}[t!]
\caption{\textbf{Q-DecPDSAG:} Decentralized Quantized and Stochastic Computation of the Semi-Discrete Entropy-Regularized Wasserstein Barycenter}
\label{Alg:mainWB}
{\small
\begin{algorithmic}[1]
  \STATE \textbf{Initialization} 
  \STATE Each agent $i\in V$ is assigned a distribution $\mu_i$.
    \STATE All agents $i \in V$ set
$\eta^{0}_i = \zeta^{0}_i= z^{0}_i = \zeta^{0}_i = \blm^{0}_i = \hat{p}^0_i = \boldsymbol{0} \in \mathbb{R}^n$, and
$\gamma = \e / (4 \ln (\Omega ))$. 
\STATE Set the number of iterations $N$
\STATE Select the sampling scheme:

 \textbf{A) Increasing batch size.} \\
Set $M_{i}^{1,k} = M_{i}^{2,k} =
\max\{1 , \lceil { ( k+ 2)}/{ \ln ( \Omega )}  \rceil \}
$,
$A_{k }= \dfr{1}{2}  (k+1)(k+2)$, $\alpha_k =  \dfr{k+1}{2}$ with
$\beta_{k}=  2L$.

  \textbf{B) Constant batch size.}\\
Set $M_{i}^{1,k} = M_{i}^{2,k}= M$,
    $\alpha_{k}= \dfr{k+1}{2\sqrt{2}}$ and $\beta_{k}= L +\dfrac{\sigma(k+2)^{3/2}}{2^{1/4}\sqrt{3} R} $,
    $A_{k}= \dfr{(k+1)(k+2)}{4 \sqrt{2}}$. 

\STATE Compute $\widetilde{\nabla} \W_{\gamma,\mu_i}^*(\blm_i^0)$ according to Lemma~\ref{lm:variance_stoch_gradient}.
\STATE Send $\widetilde{\nabla} \W_{\gamma,\mu_i}^*(\blm_i^0)$ to neighbors, i.e., $\{j \mid (i,j) \in E \}$.
\STATE  \textbf{Initialization ends.} 
\STATE  \textbf{For each agent $i \in V$:}
    \FOR{$k=0,\dots, N-1$}
        \STATE $\tau^{k} = {\alpha_{k+1}}/{A_{k+1}}$
        \STATE $z^{k}_{i} = - \dfrac{1}{\beta_{k}}   \suml{l=0}^{k} \alpha_{l} \suml{j=1}^{m} W_{ij} \widetilde{\nabla} \cW^{*}_{\gamma, \mu_{j}}(\blm_i^k)$. \label{eq:semi-discr agd summation of the gradient 1}
        \STATE $ \blm^{k+1}_{i} = \tau_{k}z^{k}_{i} + (1 - \tau_{k} )  \eta^{k}_{i}$.
        \STATE Compute $\widetilde{\nabla} \W_{\gamma,\mu_i}^*(\blm_i^{k+1})$ according to Lemma~\ref{lm:variance_stoch_gradient}.
        \STATE Send $\widetilde{\nabla} \W_{\gamma,\mu_i}^*(\blm_i^{k+1})$ to neighbors.
        \STATE $\zeta^{k+1}_{i} =  z^{k}_{i} -  \dfrac{\alpha_{k+1}}{\beta_{k}}\suml{j=1}^{m} W_{ij} \widetilde{\nabla} \cW^{*}_{\gamma, \mu_{j}}(\blm_i^{k+1})$. \label{eq:semi-discr agd summation of the gradient 2}
        \STATE $\eta^{k+1}_{i} = \tau_{k} \zeta^{k+1}_{i} + (1- \tau_{k}) \eta^{k}_{i} $.
        \STATE  $\hat{p}^{k+1}_i =
        \frac{\alpha_{k+1}p^{k+1}(\blm^{k+1}_i, \cdot) + A_k \hat{p}^{k}_i}{A_{k+1}}.$
    \ENDFOR
\ENSURE The point $\hat{p}^{N}$.
\end{algorithmic}}
\end{algorithm}

\begin{theorem}\label{th:SAGD increasing batch}
  Let $\mathcal{G}$ be a connected, undirected, fixed graph, and let $\{\mu^i\}_{i \in V}$ be a set
  of probability distributions with $\mu_i \in \M(\mathcal{Z})$. Moreover, set
  $N = \tO \l({\sqrt{\chi (\overline{W}) n} }/{\e}\r) $ (communication rounds), and pick the increasing batch size
  strategy with $ M_{i,1}^k= M_{i,2}^{k} = \max \{ 1 , \lceil { (k+ 2)}/{ \log ( \Omega ) } \rceil \} $
  for all $i \in V$ and all $k\geq 0$. Then, the output of Algorithm~\ref{Alg:mainWB}, i.e.,
  $\hat{p}^{N}$, is an $\e-$ solution of~\eqref{eq:PrimPr}. Moreover, the total number of samples drawn from $\{\mu^i\}_{i \in V}$ is
  $\tO \l( {\sqrt{\chi (\overline{W})} mn }/{ \e^{2}} \r)$, the total number of arithmetic operations is $\tO \big( \max \{ {\chi (\overline{W}) m n^{2}}/{\e^{2}},{ \chi (\overline{W}) \kappa (\overline{W}) n}/{\e^{2}}  \} \big)$ a, and the total number of bits sent per node is 
$\tO \big( \sqrt{\chi (\overline{W})} \kappa (\overline{W}) \min \{ 
    { n^{3/2} }/{\e} , {n   }/{\e^{2}} \} \big) $.
\end{theorem}

The results in Theorem~\ref{th:SAGD increasing batch} are two-fold. On the one hand, it provides the number of communication rounds required for Algorithm~\ref{Alg:mainWB} to obtain an arbitrary approximation of the non-regularized WB when sampling scheme Case A) is selected. Moreover, we provide an explicit sample, arithmetic, and communication complexities for generating such an approximation. In particular, if one uses expander graphs, which are in general well-connected, the total number of arithmetic operations is $\tO \l({ mn^{2}}/{\e^{2}}\r) $, and the total number of bits sent per node is $\tO (\min \{ { n^{3/2}}/{ \e} , { n }/{\e^{2}} \})$.

Next, we present a result for the sampling scheme Case B) in Algorithm~\ref{Alg:mainWB}.

\begin{theorem}\label{th:SAGD constant batch}
  Let $\mathcal{G}$ be a connected, undirected, fixed graph, and let $\{\mu^i\}_{i \in V}$ be a set
  of probability distributions where $\mu_i \in \M(\mathcal{Z})$. Moreover, set $N =\tO \big({\chi (\overline{W}) n }/{\e^{2}}\big) $ (communication rounds), and pick the constant batch strategy with   $ M_{i,1}^k= M_{i,2}^{k} = M $ for all $i \in V$ and all $k\geq 0$. Then, the output of   Algorithm~\ref{Alg:mainWB}, i.e., $\hat{p}^{N}$, is an $\e-$ solution of~\eqref{eq:PrimPr}. Moreover, {N}er of samples from the set of distributions
  $\{\mu^i\}_{i \in V}$ is $\tO \l( {  \chi (\overline{W})  mn }/{ \e^{2}} \r) $. The total number of arithmetic
  operations is $\tO  \big( {\chi (\overline{W}) \kappa (\overline{W}) n^{2} }/{\e^{2}} \big)$. The total number of bits sent per node for the general graph is $\tO\l({ \chi (\overline{W})   \kappa (\overline{W}) n}/{\e^{2}}\r)$.
\end{theorem}

Similarly as in Theorem~\ref{th:SAGD increasing batch}, when expander graphs are used, the total number of arithmetic   operations is $\tO \l({ mn^{2}}/{\e^{2}}\r) $, and the total number of bits sent per node is $ \tO({ mn }/{\e^{2}}) $.

\vspace{-0.2cm}
\section{Numerical Results}\label{sec:numerics}
\vspace{-0.2cm}

We present numerical experiments to support Theorem~\ref{th:SAGD constant batch} and Theorem~\ref{th:SAGD increasing batch}. Figure~\ref{fig:gaussians} shows the consensus gap and the dual function value of the WB problem generated by Algorithm~\ref{Alg:mainWB} for $m=10$ and $m=100$ agents. The set of distributions are set to be Gaussian with means and variances randomly selected and assigned. Moreover, we compare the performance for the Non-quantized AGM for WB computation proposed in~\cite{dvurechensky2018decentralize}. We test the algorithm for various combinations of sampling rates $M_1$ and $M_2$. As expected, the algorithm from~\cite{dvurechensky2018decentralize} has the best results because it is a non-quantized communication approach, and at each iteration, each agent communicates its full gradient with its neighbors. However, the experiments also show that even with a small number of samples like $M_1=1$ and $M_2=10$ Algorithm~\ref{Alg:mainWB} can rapidly converge to the barycenter with significantly less communication overhead. Figure~\ref{fig:gaussians2} shows the iterates generated by Algorithm~\ref{Alg:mainWB} for $k=\{10,100,200,500\}$ iterations. Again, we compare the performance with AGM from~\cite{dvurechensky2018decentralize} and various sampling schemes $M_1$ and $M_2$.

\vspace{-0.2cm}

\begin{figure}[t!]
    \centering
    \includegraphics[trim=9 7 12 5,clip,width=0.9\linewidth]{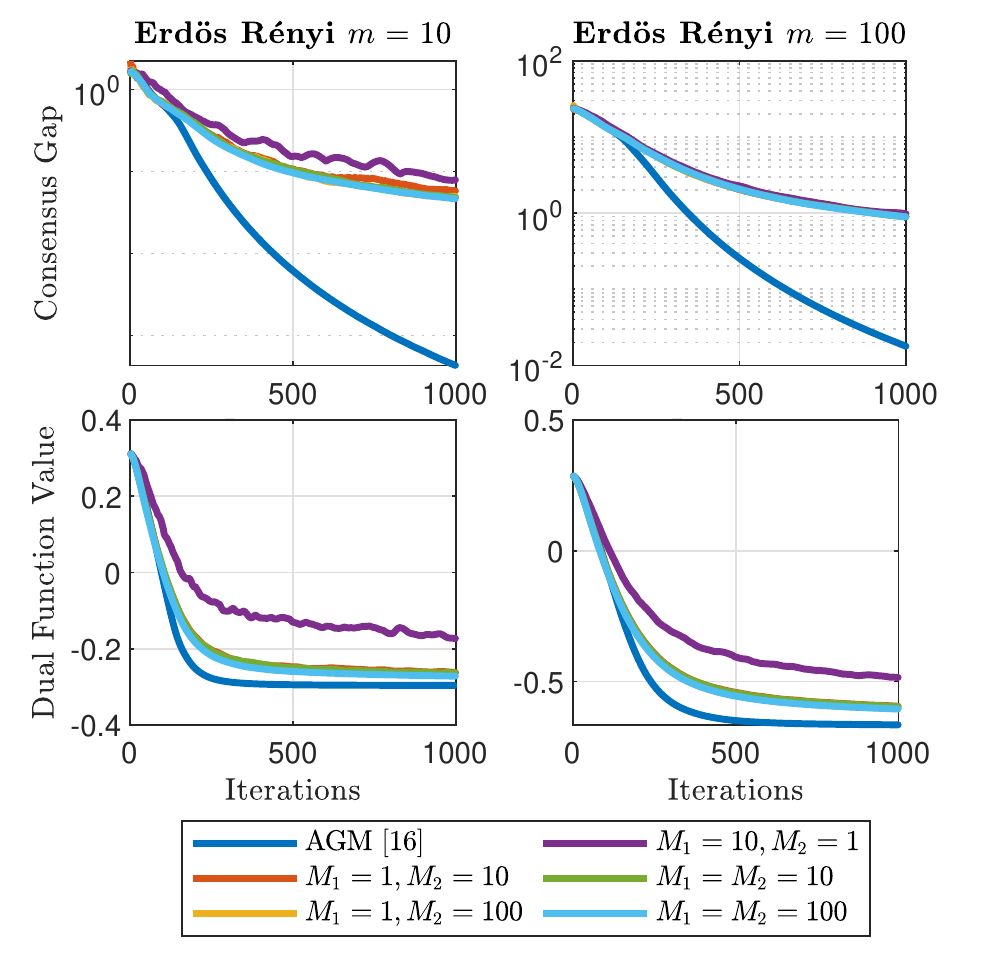}
    \caption{Consensus gap and dual function values generated by Algorithm~\ref{Alg:mainWB}.}
    \vspace{-0.2cm}
    \label{fig:gaussians}
\end{figure}

\begin{figure}[t!]
    \centering
    \includegraphics[trim=35 20 43 20,clip,width=1.01\linewidth]{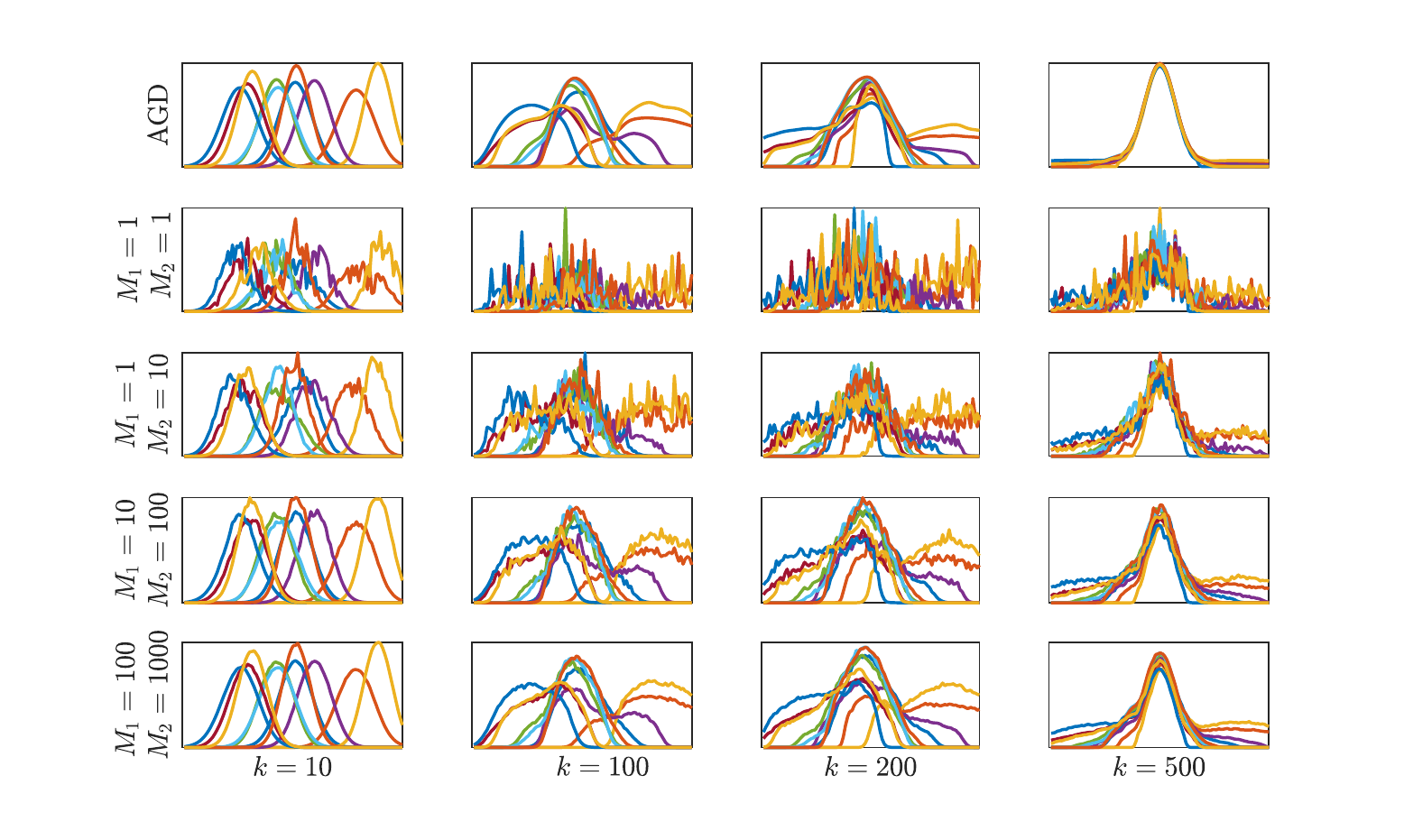}
    \caption{The iterates generated by Algorithm~\ref{Alg:mainWB}.}
    \vspace{-0.4cm}
    \label{fig:gaussians2}
\end{figure}

\vspace{-0.2cm}
\section{Discussion and Future Work}\label{sec:Discussion}

\vspace{-0.2cm}

We presented an algorithm for the computation of approximate semi-discrete WB over networks. We analyzed the primal-dual structure of the WB associated optimization problem and proposed a novel accelerated stochastic method to minimize functions with stochastic dual structure. The proposed method's main advantage is its flexibility in the sampling scheme of the dual stochastic gradient. Two sampling approaches were studied, increasing batch sized and constant batch size. The constant batch is particularly useful when a specific sample size is available at each iteration step. These sampling schemes drastically reduce the communication complexity of the WB computation. Moreover, we explicitly analyze the communication, sample, and arithmetic complexity of the proposed methods. Future work should focus on studying whether the total arithmetic operations can be improved further than ${nm^2}/{\e^2}$ as in the discrete case~\cite{kroshnin2019complexity,lin2020computational}.


\subsubsection*{Acknowledgements}
Funded by the Deutsche Forschungsgemeinschaft (DFG, German Research
Foundation) under Germany's Excellence Strategy – The Berlin Mathematics Research Center MATH+ (EXC-2046/1, project ID: 390685689). The research is supported by the Ministry of Science and Higher Education of the Russian Federation (Goszadaniye) No.075-00337-20-03, project No. 0714-2020-0005.

  \bibliographystyle{abbrv}
  \bibliography{wass,PD_references,time_varying,privacy,references11}




\onecolumn
\appendix
 \setcounter{lemma}{0}
 \setcounter{theorem}{0}
      \renewcommand{\thelemma}{\Alph{section}\arabic{lemma}}
      \renewcommand{\thetheorem}{\Alph{section}\arabic{theorem}}
      \renewcommand{\thealgorithm}{\Alph{section}\arabic{algorithm}}

\section{Proof of Proposition~\ref{prop:volume} }

It is clear that the density $\xi$ of the uniform distribution on $\Y \times \mathcal{Z}$ is $1/\Omega$, and that the entropy term in \eqref{WassDis} is bounded from below by zero and from above by the value achieved when
when $\pi$ is a Dirac mass. In the latter case the maximum value is $\gamma  \log \Omega $, which is no greater than $\varepsilon /4$ by our choice of $\gamma=\hat{\gamma} = \e / (4\log (\Omega ))$. Thus, for any $\mu,p$, $\cW_0(\mu,p) \in [\cW_{\hat{\gamma}}(\mu,p),\cW_{\hat{\gamma}}(\mu,p)+\varepsilon/4]$.
Substituting this in \eqref{eq:Precise Wasserstein Barycenter Approximation}, we obtain the desired result.

\section{Proof of Lemma \ref{lm:variance_stoch_gradient}}
By the definition of $\xi_r$, $\E_{\xi_r}[e_{\xi_r}]=\widehat{\nabla} \W_{\gamma,\mu_j}^*(\blm_j)$ and it is easy to verify that the stochastic approximation $\widetilde{\nabla} \W_{\gamma,\mu_j}^*(\blm_j)$ is unbiased, and \eqref{eq:expected value} holds.
Let us now estimate the variance.
 Denoting $\E = \E_{Y_l^i\sim \mu_j, \xi_{r}^{i} , i=1,...,m, l= 1, \dotso , M_{i,1} , r=1,...,M_{i,2}}$ where $\xi_{r}^{i}$ is distributed as described in Lemma \ref{lm:variance_stoch_gradient} we have the following chain of estimates. 
 
\begin{align*}
\E
\|\widetilde{\nabla} \W_{\gamma}^*(\Blm) - \nabla \W_{\gamma}^*(\Blm)\|_2^2 &=  \E \left\|
\sqrt{W} \left(
\begin{aligned}
\widetilde{\nabla} \W^*_{\gamma,\mu_1}&([\BBlm]_1)\\
&...\\
\widetilde{\nabla} \W^*_{\gamma,\mu_m}&([\BBlm]_m)
\end{aligned}
\right) -
\sqrt{W} \left(
\begin{aligned}
\nabla \W^*_{\gamma,\mu_1}&([\BBlm]_1)\\
&...\\
\nabla \W^*_{\gamma,\mu_m}&([\BBlm]_m)
\end{aligned}
\right)
\right\|_2^2\\
& \leq (\lambda_{\max}(\sqrt{W}))^2 \E\left\|
\begin{aligned}
\widetilde{\nabla} \W^*_{\gamma,\mu_1}([\BBlm]_1) &- \nabla \W^*_{\gamma,\mu_1}([\BBlm]_1)\\
&...\\
\widetilde{\nabla} \W^*_{\gamma,\mu_m}([\BBlm]_m) &- \nabla \W^*_{\gamma,\mu_m}([\BBlm]_m)
\end{aligned}
\right\|_2^2
\\
& = (\lambda_{\max}(\sqrt{W}))^2 \E\sum_{i=1}^m \left\| \widetilde{\nabla}\W^*_{\gamma,\mu_i}([\BBlm]_i) - \nabla \W^*_{\gamma,\mu_i}([\BBlm]_i) \right\|_2^2
\\
& = (\lambda_{\max}(W) ) \sum_{i=1}^m \E \left\| \dfr{1}{M_{i,2}} \suml{r = 1}^{M_{i,2}}  e_{\xi_{r} } (\BBlm_i, Y^i_{1} , \dotso , Y^i_{ M_{i,1}}  ) -  \E p_i (\BBlm_i, Y^i  ) \right\|_2^2.
\end{align*}


We define $\overline{p} = \dfr{1}{M_1} \s{l}{1}{M_1} p\l( \BBlm_i , Y_l^i \r)$.
Using the triangle inequality for the square of the norm we get: 
\begin{align}
\label{eq:ineq double rand}
  \E \|e_{\xi} -  \E(p) \|_2^2 \leq  2 \E \| e_{\xi} -  \bar{p} \|_2^2 + 2 \E \| \bar{p} - E(p)\|_2^2  .
\end{align}

Moreover, using the \cite[Lemma 2]{dvurechensky2018decentralize}, we have $ \E \| \bar{p} - \E(p)\|_2^2  \leq \frac{1}{M_1} -  \frac{1}{M_1}  \| \E (p) \|_2^2  $. For a
fixed $\lambda$ and $p(Y_r) := p(\lambda , Y_r)$, the first term in inequality \eqref{eq:ineq double rand} can be computed as follows:
\begin{align}
  \E \| e_{\xi} -  \bar{p} \|_2^2 = \E \left[\E[ \| e_{\xi} - \bar{p} \|_2^2 | Y] \right]  &=
                                  \E \Big[ \E \big[ \s{i}{1}{n} ([e_{\xi}]_i - [\bar{p}]_i ) | Y ] \big]
                                    \Big]  =  \E \left[ \E \left[\s{i}{1}{n} [ e_{\xi} ]_i |Y \right] \right] - \s{i}{1}{n} \E [\bar{p}]_i^2 = \nonumber \\ &= \E \left[ \s{i}{1}{n} [\bar{p}]_i \right] - \dfrac{1}{M_1^2} \s{i}{1}{n} \E \s{r}{1}{M_1} [p( Y_r) ]_i^2 \notag  = 1 - \dfrac{1}{M_1} \E \| p \|_2^2. \displaybreak
\end{align}
Using the above procedure for the sum of $e_{\xi}$ we get
\begin{align}\label{variance of stochastic gradient}
  \E \| \frac{1}{M_2} \s{r}{1}{n} e_{\xi_r} &- \E{p} \|_2^2 =  \frac{1}{M_2^2}\E \| \s{r}{1}{M_2}
\left( e_{\xi_r} - \bar{p} + \bar{p} - \E{p}\right) \|_2^2  \leq
  \E \left[ \frac{2}{M_2^2} \| \s{r}{1}{M_2} \left(  e_{\xi_r} -  \bar{p} \right)       \|_2^2 + 2 \| \bar{p} - \E{p}\|_2^2 \right].
\end{align}
Since $\xi_r$ are conditionally independent, we can bound the first term by:
\begin{align}
  \E\frac{2}{M_2^2} \| \s{r}{1}{M_2} (e_{\xi_r} - \s{r}{1}{M_2}\bar{p}) \|_2^2 \notag  = \frac{2}{M_2^2}  \Bigg( \E \Bigg[ \E \left[ \s{r}{1}{M_2} \|e_{\xi_r}\|_1 | Y \right] \Bigg] &- \notag  \E \left[ \left( \E \left[ \s{r}{1}{M_2} \| e_{\xi_r} \|_1 | Y \right] \right)^2 \right] \Bigg) \nonumber \\ & = \frac{2}{M_2^2} \E \left[ M_2 \s{i}{1}{n} [\bar{p}]_i \right] \notag - \frac{2\E \| \bar{p} \|_2^2}{M_2} =
\frac{2}{M_2} (1- E\| \bar{p}\|_2^2).
\end{align}

Thus, the inequality \eqref{variance of stochastic gradient} becomes:
\begin{align}
  \E \| \frac{1}{M_2}  \s{r}{1}{n} e_{\xi_r}  - \E{p} \|_2^2 & \leq \notag    \frac{2}{M_2} (1 - \dfrac{1}{M_1} \E \| p
\|_2^2) - \frac{2}{M_1} \| E [p] \|_2^2 + \frac{2}{M_1}  \\  & = \dfrac{2}{M_2} + \frac{2}{M_1}(1 - \| E [p] \|_2^2 -
\dfrac{2}{M_2 } \E \| p \|_2^2 ) \<= \nonumber
                                             2 \l(\dfrac{1}{M_2} + \dfrac{1}{M_1}\r) .
  \end{align}
This leads to:
\begin{align}
  \label{eq:3}
\E
\|\widetilde{\nabla} \W_{\gamma}^*(\Blm) - \nabla \W_{\gamma}^*(\Blm)\|_2^2 &=
 (\lambda_{\max}(W) ) \s{i}{1}{m} 2 \l(\dfrac{1}{M_{i,2}} + \dfrac{1}{M_{i,1}}\r),
\end{align}
and the desired result follows.

    \section{Proof of Theorem \ref{th:conv general alg}}

As mentioned in Section~\ref{sec:general}, we consider primal-dual pair of problems of the form:
\begin{align}
& \minl{ x \in Q} \{ f(x): Ax= b\}  \label{eq:pd_problems_primal} \\
  &  \minl{\lambda \in \Lambda } \vp (\lambda) = \minl{\lambda \in \Lambda }
\la \lambda, b \ra +  \max_{x\in Q}\{ -f(x) - \la A^T \lambda  ,x \ra \} =  \minl{\lambda \in \Lambda }\la \lambda, b \ra + f^*(-A^T\lambda)
    . \label{eq:pd_problems_dual}
\end{align}

As before we assume that the conditions \ref{main assumptions} are satisfied.

Algorithm \ref{Alg:APDSGD} generalizes the
 Algorithm 6.1 of \cite{devolder2011stochastic} 
 (see also \cite{dvurechensky2016stochastic,gasnikov2016stochasticInter}), here we present the Algorithm 6.1 to give a reader a better overview and its convergence theorem. Then we analyze the convergence rate of our primal-dual version of the Algorithm \ref{Alg:APDSGD}.
  Note that the primal-dual analysis of the existing accelerated methods \cite{yurtsever2015universal,anikin2017dual,bayandina2018mirror,chernov2016fast,dvurechensky2016primal-dual,dvurechensky2018computational,dvurechensky2020stable,dvurechensky2015primal-dual,guminov2019accelerated,nesterov2020primal-dual} does not apply since the dual problem is a stochastic optimization problem and we use additional randomization.

 Algorithm 6.1 of \cite{devolder2011stochastic} applied to the dual problem \eqref{eq:pd_problems_dual} with stochastic inexact oracle $\tgrad \Phi (\lambda, \xi, \wdxi )$ is listed as Algorithm~\ref{Alg:ASGD}. Since the objective $\vp(\lambda)$ has Lipschitz-continuous gradient, in the setting of \cite{devolder2011stochastic} we have $\delta=0$.
 

\begin{algorithm}[!h]
\caption{Stochastic Fast Gradient Method (\cite[Algorithm 6.1]{devolder2011stochastic})}
\label{Alg:ASGD}
{\small
\begin{algorithmic}[1]
\REQUIRE Sequences $\alpha_k,\beta_k,A_k=\sum_{i=0}^k\alpha_k$ satisfying Assumptions \ref{coefficients condition}. Distance generating function $d(\lambda)$ and corresponding Bregman divergence $V[z_{k}](\lambda)$.
    \STATE  \textbf{Initialization}\\
  Set $\eta_{0} = \zeta_{0} = z_{0}  =  \lambda_{0}  $ and choose a number of iterations $N$.
    \FOR{$k=0,\dots, N-1$}
        \STATE $ \tau_{k} = {\alpha_{k+1}}/{A_{k+1}}$.
        \STATE \begin{equation}
          \label{eq:z_step}
        z_{k}= \argmin_{\lambda \in \Lambda} \{ \beta_{k} d(\lambda ) + \suml{l=0}^{k} \alpha_{l} \la \tgrad \Phi(\lambda_{l}, \xi_{l},\wdxi_{l} ), \lambda - \lambda_{l} \ra \}
        \end{equation}
        \STATE $\lambda_{k+1} = \tau_{k}z_{k} + (1 - \tau_{k} )  \eta_{k}$.
        \STATE Compute $\tgrad \Phi( \lambda_{k+1}, \xi_{k+1}, \wdxi_{k+1})$.
        \STATE
        \begin{equation}
          \label{eq:zeta_step}
        \zeta_{k+1}  = \argmin_{\lambda \in \Lambda}\{ \la \beta_{k}V[z_{k}](\lambda) + \alpha_{k+1} \la \tgrad\Phi( \lambda_{k+1}, \xi_{k+1}, \wdxi_{k+1} ) , \lambda - z_{k} \ra
                      \} .
        \end{equation}
        \STATE $\eta_{k+1} = \tau_{k} \zeta_{k+1} + (1- \tau_{k}) \eta_{k} $.
        \ENDFOR
    \ENSURE The point $\eta_{N}$.
\end{algorithmic}}
\end{algorithm}

The following condition for the convergence analysis of Algorithm~\ref{Alg:ASGD} is stated at the beginning of Section 6.1 in \cite{devolder2011stochastic}.
\begin{assumptions}
  \label{coefficients condition}
  The following conditions hold for the sequences $\{\alpha_k\}$ and $\{\beta_k\}$, for all $k\geq 0$.
\begin{align}
  \label{eq:condition on the coefficients AGD}
  \alpha_{0} \in ]0,1 ]  \text{ and } & \beta_{k+1} \>= \beta_{k} > L \text{ for all } k\>= 0 , \notag \\
 \text{Coupling condition: } & \alpha_{k}^{2}\beta_{k} \<= \Big(\suml{l=0}^{k} \alpha_{\ell} \Big) \beta_{k-1}.
\end{align}
\end{assumptions}

We denote, for any $N\geq 0$,
  \begin{align}
     \widetilde{\Psi}_{N}(\lambda)  &= \beta_{N}d(\lambda) + \suml{k=0}^{N} \alpha_{k} \big( \Phi (\lambda_{k}, \xi_{k}) \notag
     + \la \tgrad \Phi (\lambda_{k} , \xi_{k}, \wdxi_{k}) , \lambda - \lambda_{k} \ra \big) , \quad \wdPsi_{N}^{*} = \minl{\lambda \in \Lambda} \wdPsi_{N}(\lambda).
  \end{align}

    \begin{theorem}[Lemma 2 in \cite{devolder2011stochastic}]\label{th:prec_cert}
Assume that Conditions~\ref{main assumptions}, \ref{coefficients condition} are satisfied. Then, the output of Algorithm~\ref{Alg:ASGD} 
for all $N\>= 0$ has the following property:
    \begin{align}
\label{ineq:prec_cert}
     A_{N}\vp(\eta_{N}) &\<= \widetilde{\Psi}_{N}^{*} + \suml{k=0}^{N} \alpha_{k} ( \vp (\lambda_{k}) - \Phi(\lambda_{k}, \xi_{k}))
     + \suml{k=1}^{N} A_{k-1} \la \nabla \vp(\lambda_{k}) - \tgrad \Phi(\lambda_{k}, \xi_{k}, \wdxi_{k}), \lambda_{k} - \eta_{k-1} \ra \notag \\ & \qquad  + \suml{k=0}^{N} \dfrac{A_{k}}{(\beta_{k}- L )} \|\tgrad \Phi_{k}(\lambda_{k},\xi_{k},\wdxi_{k}) -\nabla \vp (\lambda_{k}) \|_{2}^{2}.
    \end{align}

  \end{theorem}
  
Let us now move to our generalization of Algorithm~\ref{Alg:ASGD}.  In our setting, we take Algorithm~\ref{Alg:ASGD} with a particular choice of the prox-function $d(\lambda)= \dfr{1}{2} \|\lambda \|_{2}^{2}$ and the corresponding Bregman divergence
$V[\wdlm](\lambda) = \dfr{1}{2} \| \lambda - \wdlm \|_{2}^{2}$.
Then, the steps \eqref{eq:z_step} and \eqref{eq:zeta_step} of Algorithm~\ref{Alg:ASGD} can be computed analytically. Indeed, the optimality conditions result in equations
$ \beta_{k} \lambda + \suml{l=0}^{k} \alpha_{l} \tgrad \Phi (\lambda_{l}, \xi_{l},\wdxi_{l})= 0$ and
$  \beta_{k} ( \lambda - z_{k} ) + \alpha_{k+1}  \tgrad \Phi (\lambda_{k+1} , \xi_{k+1}, \wdxi_{k+1})= 0$. 
The steps therefore read as
$ z_{k} = - \dfrac{1}{\beta_{k}}   \suml{l=0}^{k} \alpha_{l} \nabla \Phi (\lambda_{l}, \xi_{l})  $ and
$ \zeta_{k+1} =  z_{k} -  \dfrac{\alpha_{k+1}}{\beta_{k}}  \nabla \Phi (\lambda_{k+1} , \xi_{k+1})$.
Our generalization consists in adding an update for the primal variable, such that not only the dual, but also the primal problem \eqref{eq:pd_problems_primal} is solved with an optimal rate.
Thus, in our primal-dual algorithm, we add the step
$ \hat{x}_{k+1} = \frac{1}{A_{k+1}}\sum_{i=0}^{k+1} \alpha_i x(-A^T\lambda_i,\xi_i) = \frac{\alpha_{k+1}x(-A^T\lambda_{k+1},\xi_{k+1})+A_k\hat{x}_{k}}{A_{k+1}}$
to the algorithm
for the update of the primal variable. Here the vector $x(-A^T\lambda,\xi)$ is defined as
\begin{align*}
x(-A^T\lambda,\xi) = \arg\max\limits_{x\in Q}\{\la -A^T\lambda,x\ra - F(x,\xi)\}.
\end{align*}
We also specify two choices of sequences $\alpha_k,\beta_k,A_k=\sum_{i=0}^k\alpha_k$.
Having made these adjustments we arrive at Algorithm~\ref{Alg:APDSGD}.






The following Lemma is needed for the proof of Theorem \ref{th:agd_conv_rate_primdual}.
\begin{lemma}
\label{lm:primaldualineqforlinearpart}
Let the assumptions of Section~\ref{sec:general} in the main paper hold. Then
\begin{align}
  \label{eq:proof_st_3}
  \E_{\xi_k} (\vp(\lambda_{k}) &+ \la \nabla \Phi(\lambda_{k},\xi_k), \lambda - \lambda_{k}\ra ) \leq -f(\E_{\xi_k}x(-A^T\lambda_k,\xi_k)) + \la b -
A\E_{\xi_k} x(-A^T\lambda_k,\xi_k), \lambda\ra.
\end{align}
\end{lemma}
\begin{proof}
The proof can be found in~\cite[Theorem 2]{dvurechensky2018decentralize}.
\end{proof}

Theorem \ref{th:agd_conv_rate_primdual} gives a convergence result for Algorithm \ref{Alg:APDSGD} which is a primal-dual version of the Algorithm \ref{Alg:ASGD}. 

  \begin{theorem}\label{th:agd_conv_rate_primdual}
Let Assumptions \ref{main assumptions} hold 
and the variance in each iteration be bounded $\E_{\xi_{k},\wdxi_k} \| \tgrad \Phi (\lambda_k,\xi_k,\wdxi_k ) - \nabla \vp(\lambda_k) \|_2^2 \leq \sigma^{2}_k$.
Let accuracy $\varepsilon > 0$ and the number of steps $N \geq 0$ be given.
Then, for the primal sequence $\hat{x}_{k+1} = \frac{1}{A_{k+1}}\sum_{i=0}^{k+1} \alpha_i
x(-A^T\lambda_i,\xi_i)  = \frac{\alpha_{k+1}x(-A^T\lambda_{k+1},\xi_{k+1})+A_k\hat{x}_{k}}{A_{k+1}}$
its expected value $\E \hat{x}_{N}$ satisfies 
\begin{align}
  f(\mathbb{E}\hat{x}_N)-f^* \leq
     \dfr{ \beta_{N}R^{2}}{2 A_{N}}  + \dfrac{1}{A_{N}} \suml{k=0}^{N} \dfrac{A_{k}}{(\beta_{k}- L )} \sigma^{2}_k , \label{eq:APDSGD prec cert}
\\ \|A\mathbb{E}\hat{x}_N-b\|_2 \leq
     \dfr{ \beta_{N}R}{2 A_{N}}  + \dfrac{1}{R A_{N}} \suml{k=0}^{N} \dfrac{A_{k}}{(\beta_{k}- L )} \sigma^{2}_k . \label{eq:APDSGD prec cert 2}
\end{align}
where the expectation is taken w.r.t. all the randomness $\xi_1,\dots, \xi_N$, $\wdxi_1,\dots, \wdxi_N$.

  \end{theorem}
  \begin{proof}
    The dualization result relies on the weak duality and convexity of the considered problem.
 Taking the expectation in the inequality~\eqref{ineq:prec_cert} in Theorem~\ref{th:prec_cert} for the unbiased stochastic
 approximation cancels out certain terms and we obtain
 \begin{equation}
   \label{eq:phi iter ineq}
  A_{N} \vp (\eta_{N}) \<= \E \wt{\Psi}_{N}^{*}+   \suml{k=0}^{N} \dfrac{A_{k}}{(\beta_{k}- L )} \E \|\tgrad \Phi_{k}(\lambda_{k},\xi_{k},\wdxi_{k}) -\nabla \vp (\lambda_{k}) \|_{2}^{2}.
 \end{equation}
  Let us first estimate $\E \wtPsi_{N}^{*}$.
We introduce the feasible set $\Lambda_{R}:= \{ \lambda \in H^{*}: \|\lambda\|_{2} \<= 2R \}$. Then,

  \begin{align}
    \label{eq:7}
\E   \minl{\lambda \in \Lambda} \wtPsi_{N}(\lambda) &=
\E  \Big[ \minl{\lambda \in \Lambda}
    d(\lambda)  +  \suml{k=0}^{N} \alpha_{k} ( \vp(\lambda_{k})  +  \la \tgrad \Phi (\lambda_{k} , \xi_{k}, \wdxi_{k}) , \lambda - \lambda_{k} \ra ) \Big] \notag \\ \notag
    &\<=    \minl{\lambda \in \Lambda_{R}}
      \beta_{N}d(\lambda)  + \E \Big[  \suml{k=0}^{N} \alpha_{k} ( \vp (\lambda_{k}) + \la \tgrad \Phi (\lambda_{k} , \xi_{k}, \wdxi_{k}) , \lambda - \lambda_{k} \ra ) \Big] \\
                                            &=
     \minl{ \lambda \in \Lambda_{R} } \beta_{N}d(\lambda)  +  \suml{k=0}^{N} \alpha_{k} \E \Big[ ( \vp (\lambda_{k}) + \la \tgrad \Phi (\lambda_{k} , \xi_{k}, \wdxi_{k}) , \lambda - \lambda_{k} \ra ) \Big]
  \end{align}
  Applying Lemma \ref{lm:primaldualineqforlinearpart} above and using $\E_{\xi_k,\wdxi_k} \tgrad \Phi (\lambda_{k} , \xi_{k}, \wdxi_{k}) =\E_{\xi_k} \nabla \Phi (\lambda_{k} , \xi_{k}) $, we get:
  
  \begin{align*}
& \minl{\lambda \in \Lambda_{R}}   \suml{k=0}^{N} \alpha_{k} \E \Big[ ( \vp (\lambda_{k}) + \la \tgrad \Phi (\lambda_{k} , \xi_{k}, \wdxi_{k}) , \lambda - \lambda_{k} \ra ) \Big]  \\
& =\minl{\lambda \in \Lambda_R}   \suml{k=0}^{N} \alpha_{k}  \E \Big[ ( \vp (\lambda_{k}) + \la \nabla \Phi (\lambda_{k} , \xi_{k}) , \lambda - \lambda_{k} \ra ) \Big] \\ & \<=
 \minl{\lambda \in \Lambda_R} \left\{ \sum_{k=0}^N \alpha_{k} (-f(\E x(-A^T\lambda_k,\xi_k)) +  \la b - A\E  x(-A^T\lambda_k,\xi_k), \lambda\ra ) \right\}  \notag \\
&\leq
A_N  \minl{\lambda \in \Lambda_{R} } \left\{ -f(\E \hat{x}_N) +  \la b - A\E \hat{x}_N, \lambda\ra \right\} \\ &= -A_N f(\E \hat{x}_N) + A_N \min_{\lambda \in \Lambda_R} \la b - A\E \hat{x}_N, \lambda\ra \notag \\
    &= -A_N f(\E \hat{x}_N) +   A_N \min_{\lambda \in \Lambda_R} \la b - A\E \hat{x}_N,
        \dfrac{A\E \hat{x}_N - b}{\| A\E \hat{x}_N - b \|    } \lambda\ra \notag \\
  &=    -A_N f(\E \hat{x}_N)    - 2 A_N R\| b - A\E  \hat{x}_N \|_2.
  \end{align*}
  In the inequalities above, we also used the convexity of $f$, the property of $A_{N}$ being a linear combination $A_{N}= \suml{k=0}^{N} \alpha_{k}$, and definitions of $\hat{x}_{N}$ and $\Lambda_{R}$.
  
  Thus, $\E \wtPsi^{*}_{N}$ is bounded by
  \begin{equation}
    \label{ineq:Psi*bound}
    \E \wtPsi^{*}_{N} \<=  \beta_{N}d(\lambda) - A_N f(\E \hat{x}_N)    - 2 A_N R\| b - A\E  \hat{x}_N \|_2.
  \end{equation}
  
  Moreover, using \eqref{eq:phi iter ineq}, and  \eqref{ineq:Psi*bound}, the fact that $d(\lambda ) \<= R^{2}/2$, and the unbiasedness of the stochastic approximation, we get
  the following bound on the
  iteration term $A_{N} \vp(\eta_{N} )$:
  \begin{align}
    \label{ineq:ineq_dual_iterate}
    A_{N} \E \vp(\eta_{N} ) \<=
    \beta_{N}R^{2}/2 &- A_N f(\E \hat{x}_N)    - 2 A_N R\| b - A\E  \hat{x}_N \|_2 +
     \suml{k=0}^{N} \dfrac{A_{k}}{(\beta_{k}- L )}  \E \| \tgrad \Phi (\lambda_{k} , \xi_{k}, \wdxi_{k})-\nabla \vp (\lambda_{k}) \|_{2}^{2}.
  \end{align}
  
Now, let us estimate the convergence of the primal objective $f(\E \hat{x}_{N})$. 

Using the weak duality $-f(x^{*}) \<= \vp (\lm^{*})$ we obtain:
\begin{align}\label{eq:func_math_fin}
	f(\E\hat{x}_N)  - f(x^*) \leq f(\E\hat{x}_N) + \vp(\eta^*) \leq f(\E\hat{x}_N) + \E \vp(\eta_N)
\end{align}
Dividing all terms by $A_{N}$ in inequality \ref{ineq:ineq_dual_iterate} and using it we obtain
\begin{align}\label{eq:conv_rate_primal}
   f(\E\hat{x}_N)  - f(x^*) &\leq f(\E\hat{x}_N) + \E \vp(\eta_N) \\
  &\<=
     \dfr{ \beta_{N}R^{2}}{2 A_{N}} - 2 R\| b - A\E \hat{x}_N \|_2 + \dfrac{1}{A_{N}} \suml{k=0}^{N} \dfrac{A_{k}}{(\beta_{k}- L )} \E \|\tgrad \Phi (\lambda_{k} , \xi_{k}, \wdxi_{k}) -\nabla \vp (\lambda_{k}) \|_{2}^{2} \\
  &=
     \dfr{ \beta_{N}R^{2}}{2 A_{N}}  + \dfrac{1}{A_{N}} \suml{k=0}^{N} \dfrac{A_{k}}{(\beta_{k}- L )}  \E \|\tgrad \Phi (\lambda_{k} , \xi_{k}, \wdxi_{k}) -\nabla \vp (\lambda_{k}) \|_{2}^{2}.
\end{align}

Regarding the convergence of the argument $\| A \E \hat{x}_{N} - b \|_{2}$, we know
\begin{align}
  \label{ineq:estimate_for_arg_conv}
  2 R\| b - A\E \hat{x}_N \|_2   \leq & f(x^{*}) - f(\E\hat{x}_N) +  \dfr{ \beta_{N}R^{2}}{2 A_{N}} + \dfrac{1}{A_{N}} \suml{k=0}^{N} \dfrac{A_{k}}{(\beta_{k}- L )} \E \| \tgrad \Phi (\lambda_{k} , \xi_{k}, \wdxi_{k}) -\nabla \vp (\lambda_{k}) \|_{2}^{2}.
  \end{align}

Therefore, we have for any $x\in Q$, $f(x^{*}) \<= f(x) + \la \lambda^{*}, Ax -b \ra $. Using the bound $\| \lambda^{* } \|_{2} \<= R$,
and choosing $x = \E \hat{x}_{N}$ translates to
$f(x^{*}) \<= f(\E \hat{x}_{N}) + \la \lambda^{*}, A\E \hat{x}_{N} -b \ra  \<=
f(\E \hat{x}_{N}) + R \| A\E \hat{x}_{N} -b \|_{2} $. Finally, using this estimate in the inequality
\eqref{ineq:estimate_for_arg_conv} gives us:
\begin{align}
  \label{eq:10}
  2 R\| b - A\E \hat{x}_N \|_2  & \leq  R\|A \E \hat{x}_{N} - b\|_{2} + \dfr{ \beta_{N}R^{2}}{2 A_{N}} + \dfrac{1}{A_{N}} \suml{k=0}^{N} \dfrac{A_{k}}{(\beta_{k}- L )} \E \| \tgrad \Phi (\lambda_{k} , \xi_{k}, \wdxi_{k}) - \nabla \vp (\lambda_{k}) \|_{2}^{2},
\end{align}
and the desired result follows.

    \end{proof}
Using the techniques developed in \cite{dvurechensky2018parallel,dvinskikh2019primal}, these convergence rate guarantees can be extended to the bounds in terms of probability of large deviations.

We next analyze Case A) of Algorithm~\ref{Alg:APDSGD} and prove Theorem \ref{th:conv general alg} case A).
\begin{theorem}[Proof of Theorem \ref{th:conv general alg} case A] \label{th:agd_dec_var_conv_rate}
Select Case A in Algorithm~\ref{Alg:APDSGD} by choosing $A_{k }= \dfr{1}{2}  (k+1)(k+2)$, $\alpha_k =  \dfr{k+1}{2}$ with
$\beta_{k}=  2L$ .
Assume that the variance $\sigma_k$ satisfies $\sigma_k^2 \<= \dfr{\e L \alpha_k}{A_k}$.
 Then, Algorithm~\ref{Alg:APDSGD} gives an $\e$-solution after $  \sqrt{\dfr{8LR^{2}}{\e}}  $ iterations. 
\end{theorem}
\begin{proof}
By \eqref{eq:APDSGD prec cert}, we obtain
\begin{align}
  \label{eq:13}
   f(\mathbb{E}\hat{x}_N)-f^*
  &=   \vp(\eta_{N}) - \vp(\lambda^{*}) \<= \dfr{\beta_{N} R^{2}}{2 A_{N}} + \dfr{1}{A_{N}} \suml{k=0}^{N}\dfr{A_{k}}{\beta_{k}- L} \sigma_{k}^{2}  \notag \\ \notag
 &\<= \dfr{4LR^{2}}{ (N+1)(N+2)} + \dfr{4\e}{ L (N+1)(N+2) } \suml{k=0}^{N}\dfr{A_{k}}{ L} \dfr{\e L \alpha_k}{A_k} \\
 &=
 \dfr{4LR^{2}}{ (N+1)(N+2)} + \dfr{4\e}{ L (N+1)(N+2) } \suml{k=0}^{N} L \dfr{k+1}{4} \notag \\ \notag
 &=
\dfr{4LR^{2}}{ (N+1)(N+2)} + \dfr{2\e }{   (N+1)(N+2) } \dfr{(N+1)(N+2)}{4}  \notag \\
 &=
 \dfr{4LR^{2}}{ (N+1)(N+2)} + \dfr{\e }{2}.
\end{align}
 We can see that after $N=  \sqrt{\dfr{8LR^{2}}{\e}}$ the r.h.s. becomes smaller than $ \dfr{\e}{2}$. The bound for $\|A\E\hat{x}_N-b\|_2$ is obtained in the same way using \eqref{eq:APDSGD prec cert 2}.

%
\end{proof}

We next analyze Case B) of Algorithm~\ref{Alg:APDSGD} and prove Theorem \ref{th:conv general alg} case B).

\begin{theorem}[Proof of Theorem \ref{th:conv general alg} case B]
~\label{th:agd_const_var_conv_rate}
  Assume that we have some constant variance $\sigma_{k}^{2 } = \sigma^{2} $.
  Given $c\>= 1$, $a=2^{c}$, $b>0$ let the coefficients be defined by
  $\alpha_{k} = \dfrac{k+1}{2^{c}}$ and $\beta_{k}= L + \dfrac{\sigma}{R} b (k+2)^{c}$,
  $A_{k}= \suml{\ell=0}^{k} \alpha_{\ell} = \dfr{1}{2a}(k+1)(k+2)$.
  Select case B of Algorithm~\ref{Alg:APDSGD} by setting $c=3/2$ with the resulting step sizes
  $\alpha_{k}= \dfr{k+1}{2\sqrt{2}}$ and $\beta_{k}= L +\dfrac{\sigma}{2^{1/4}\sqrt{3} R} (k+2)^{3/2}$.
  Then Algorithm~\ref{Alg:APDSGD} gives an $\e$-solution after
$  \sqrt{\dfr{4 LR^{2}}{\e}} \bigvee \dfr{9 \sigma^{2} R^{2}}{\e^{2}}$ iterations.
\end{theorem}

\begin{proof}

  Using \eqref{eq:APDSGD prec cert}, we get
  \begin{align}
    \label{eq:conv_rate_agd_const_batch}
   f(\mathbb{E}\hat{x}_N)-f^* & \<=
    \dfr{\beta_{N} R^{2}}{2 A_{N}}  +  \dfr{1}{A_{N}} \suml{k=0}^{N}\dfr{A_{k}}{\beta_{k}- L} \sigma_{k}^{2}  \notag \\
   & \<=
     \dfr{2^{3/2}LR^{2}}{(N+1)(N+2)} + \dfr{2^{9/4}(N+3)^{3/2} \sigma R}{\sqrt{3}(N+1)(N+2)} \notag \\
    & =
   \dfr{4LR^{2}}{N^{2}} + \dfr{\sqrt{3}(3^{1/3} N)^{3/2} \sigma R}{N^{2}} \notag \\
                                                                 & =
   \dfr{4LR^{2}}{N^{2}} + \dfr{ 3  \sigma R}{ N^{1/2}  } \notag \\
  &=    \Theta \Big(  \dfrac{LR^{2}}{N^{2}} + \dfrac{\sigma R}{\sqrt{N}} \Big).
  \end{align}
By the choice of the number of iterations $ N=  \sqrt{\dfr{4 LR^{2}}{\e}} \bigvee \dfr{9 \sigma^{2} R^{2}}{\e^{2}}$, we obtain that the r.h.s. is smaller than $\varepsilon$. The bound for $\|A\E\hat{x}_N-b\|_2$ is obtained in the same way using \eqref{eq:APDSGD prec cert 2}.

\end{proof}

\section{Proof of Theorem~\ref{th:SAGD increasing batch} }

\begin{proof}
 For the Wasserstein barycenter problem, the dual objective function is  $\vp (\Blm )= \cW^{*}_{\gamma}( \Blm )$, the primal objective is $f( p ) =
    \frac{1}{m}\sum\limits_{i=1}^{m}  \W_{\gamma, \mu_i}(p_i) $, and the  linear operator is $A = \sqrt{W}$.
We set the approximate gradient $\tgrad \Phi(\Blm , Y_{1}, \dotso, Y_{M_{1}}, \xi_{1}, \dotso , \xi_{M_{2}}) := \widetilde{\nabla} \W_{\gamma}^*( \Blm, Y_{1}, \dotso, Y_{M_{1}}, \xi_{1}, \dotso , \xi_{M_{2}})$.
With this setting we can apply the Theorem~\ref{th:agd_dec_var_conv_rate}.
We set $ \gamma = \e / (4 \ln ( \Omega ))$ and the weights to $\omega_l =\dfr{1}{m}$.
According to Theorem~\ref{th:agd_dec_var_conv_rate} the variance has to satisfy the condition
\begin{equation}
\label{eq:var_bound}
\sigma_k^2 \<= \dfr{\e L \alpha_k}{A_k} =  \dfr{\e m\eigmax  }{  (k +2 ) \gamma }  =
\dfr{   \eigmax  4 m   \ln (\Omega )}{ (k+2) },
\end{equation}
where we used also Lemma \ref{Lm:dual_obj_properties2}.
According to Lemma \ref{lm:variance_stoch_gradient} the total variance is bounded by
\begin{align}
  \E \|\widetilde{\nabla} \W_{\gamma}^*(& \Blm)   - \nabla \W_{\gamma}^*(\Blm)\|_2^2 \leq
2  \lambda_{\max}(W)   \sum_{i=1}^m \Big( \dfr{1}{M_{i,1}} + \dfr{1}{M_{i,2}} \Big)
 , \; \Blm \in \R^{mn}.
\end{align}
Setting the batch sizes $M_{k} =  1 \bigvee\Big \lceil  \dfr{  k+ 2}{ \ln(\Omega) }   \Big  \rceil \<=
1 \bigvee\Big \lceil  \dfr{  k}{ \ln(\Omega) }  +1 \Big \rceil $ which is the Case A of Algorithm~\ref{Alg:mainWB} satisfies the condition
\eqref{eq:var_bound}, and thus we can apply Theorem~\ref{th:agd_dec_var_conv_rate}.
Since the decreasing variance condition is satisfied
the algorithm converges after $N =  \sqrt{\dfr{8LR^{2}}{\e}}$ iterations
according to Theorem~\ref{th:agd_dec_var_conv_rate}. Using the estimate
$R \<= \sqrt{ \dfr{2n }{m  \lambda_{\min}^+ (W)} }$
from~\eqref{eq:estimate of R}, and
$L= m \lambda_{\max}(W) /
\gamma $ from Lemma~\ref{Lm:dual_obj_properties2} we have that Algorithm~\ref{Alg:mainWB} converges after
$  \tO \Big(  \dfr{  \sqrt{\chi (\overline{W}) n}}{\e} \Big)$ iterations.

Moreover, the total number of samples is
\begin{align}\label{eq:number of samples AGD semi discrete decr}
m  \sum_{k=1}^{N}    M_k  &\<=
m \suml{k=1}^{N} k+2
=
m \Big( \dfr{N (N+1)}{2} +2N \Big) \notag
\\ &= \tO \l(\chi (\overline{W}) m N^{2}\r) = \tO \l( \dfr{\chi (\overline{W}) m n}{\e^{2}}  \r) .
\end{align}

Let us compute the
computational complexity for a particular node $i$.
On each iteration we have to batch $M_{k}$ samples. Sampling individual vector sample
$p_i(\blm_i, Y_r^i)$ costs $n$ with
$ r = 1 , \dotso, M_{k}$. Therefor the overall sampling cost is $M_{k}n$ at the $k$-th iteration.

At iteration $k$ the cost of summation of individual gradients
(Equations~\eqref{eq:semi-discr agd summation of the gradient 1} and~\eqref{eq:semi-discr agd summation of the gradient 2} in Algorithm~\ref{Alg:mainWB})
is $\nnz M_{k}$.

Obtaining realizations costs $n$. Other operations cost $O(n)$, thus the computational complexity is $\tO ( M_{k} n ) $ at each iteration.
Summing the complexity up over the iterations results in
\begin{align}
   \sum_{k=1}^{N}  (  n \log_2 n +  M_k (n  + \nnz)  ) &=
  N n \log_2 n  + (n + \nnz )\suml{k=1}^{N} \dfr{k+2}{\ln(\Omega)}
                                                       \notag   \\
&\<=
  N n \log_2 n  + (n + \nnz )\suml{k=1}^{N} k+2
\notag \\
 & =
   N n\log_2 n
    +
    (n + \nnz ) \Big( \dfr{N (N+1)}{2} +2N \Big)
\end{align}
computational complexity per node. Summing this up over all the nodes results in
\begin{align}
& \s{i}{1}{m} N n\log_2 n
  +  (n + \nnz ) \Big( \dfr{N (N+1)}{2} +2N \Big) \notag \\
  &=
m N n\log_2 n
    +  (m n +
 \s{i}{1}{m} \nnz
      ) \Big( \dfr{N (N+1)}{2} +2N \Big) \notag \\
  &=
    \tO \l( m n N^{2}  \bigvee   \s{i}{1}{m} \nnz N^{2} \r) = \tO \l( \dfr{\chi (\overline{W}) m n^{2}}{\e^{2}} \bigvee
\dfr{\chi (\overline{W}) n}{\e^{2}} \kappa (W)
    \r).
\end{align}
\par
Compared to the non-quantized version the number of bit communications per node per round for the quantized version is now $\nnz(M_{k} \bigwedge n )$ instead of $\nnz \cdot n$ because we send $M_k$ coordinates and at most the full dimensional vector.
Summing this
over the number of iterations and over the nodes gives us
\begin{align}  \s{i}{1}{m} & \sum_{k=1}^{N}  \nnz ( n \wedge  M_{k} )   \<=
  \s{i}{1}{m} \nnz \s{k}{1}{N}   n \bigwedge \s{k}{1}{N} M_{k}   =
  \tO\l(
\sqrt{\chi (\overline{W})} \kappa( \overline{W} ) \l(
     \dfrac{ n^{3/2} }{\e} \bigwedge   \dfr{n   }{\e^{2}} \r)
  \r) 
\end{align}
total number of sent communication bits. \newline

\end{proof}

\begin{remark}
  The complexity bounds for the expander graph are easily derived from the general results of
Theorem~\ref{th:agd_dec_var_conv_rate} and~\ref{th:agd_const_var_conv_rate} using the equality
$ \s{i}{1}{m} \nnz = \tO(m)$.
\end{remark}

\section{Proof of Theorem~\ref{th:SAGD constant batch}}
\label{sec:proof case B WB}
\begin{proof}
We choose the dual, primal objective and approximate gradient as in proof of Theorem \ref{th:SAGD increasing batch}.
According to Theorem~\ref{th:agd_const_var_conv_rate}, Algorithm~\ref{Alg:mainWB} case B) converges after
$ N \>=  \sqrt{\dfr{4 LR^{2}}{\e}} \bigvee \dfr{9 \sigma^{2} R^{2}}{\e^{2}}$ iterations.
We choose case B of Algorithm \ref{Alg:mainWB} with constant batch size
$M = M_{i,1}^k= M_{i,2}^{k} $.
Then, by Lemma \ref{lm:variance_stoch_gradient}
\begin{align}
\sigma^{2}\leq
2  \lambda_{\max}(W)   \sum_{i=1}^m  \dfr{2}{M} = \dfr{ 4 \eigmax m}{M}.
\end{align}
Using again the estimate for $R$ and $L$ the algorithm converges after
$ N =
\sqrt{\dfr{32 \chi (\overline{W})n}{\e}} \bigvee \dfr{72   \chi (\overline{W}) n}{ \e^{2}M} =
  \tO \Big(  \dfr{\chi(\overline{W}) n}{\e^{2}} \Big)$ iterations for a small constant batch size.

The total number of samples is
$ m  \sum_{k=1}^{N}    M  =
m N M = \tO \l( \dfr{\chi(\overline{W}) m n}{\e^{2}}  \r) . $

Sampling individually $p_i(\blm_i, Y_r^i) ,\quad r = 1 , \dotso, M$ costs $n$. Therefor the overall sampling cost is $Mn$ for each iteration.

At iteration $k$ the cost of summation of individual gradients
(Equations~\eqref{eq:semi-discr agd summation of the gradient 1} and~\eqref{eq:semi-discr agd summation
of the gradient 2} in Algorithm~\ref{Alg:mainWB})
is
$ \nnz \cdot M$.
Sampling $M$ type 2 samples costs $M n$. Afterwards the cost of sampling $M$ type 1 samples is also $Mn$.
Other operations cost $O(n)$, thus the computational complexity is $\tO ( M n ) $ at each iteration.
Summing the complexity up over the overall computational complexity per node is
\begin{align}
\sum_{k=1}^{N}    M n   =
\tO \l( \dfr{\chi(\overline{W}) n^{2} }{\e^{2}}  \r)
\end{align}
resulting in
$ \sum_{i=1}^{m}    \sum_{k=1}^{N}    M n  \cdot \nnz =
\tO \l( \dfr{\chi (\overline{W}) n^{2}  }{\e^{2}}  \r)  $ total arithmetic operations.

The number of bits sent per iteration for the quantized version is now
$ \nnz \cdot (M \wedge n)= \nnz \cdot M $ per node for $M \<= n$. Summing these estimates over the number of iterations gives us
\begin{align}
&\s{i}{1}{m}  \sum_{k=1}^{N} \nnz \cdot M   \notag \\ &=
 \s{i}{1}{m} \nnz M \l(
  \sqrt{\dfr{32 n\chi (\overline{W} )}{\e}} \bigvee \dfr{72  n \chi (\overline{W} ) }{ \e^{2}M} \r)
  =
  \tO\l(
                                                        \dfr{\kappa(\overline{W}) \chi (\overline{W} ) n }{\e^{2}}
  \r)
\end{align}
number of communication bits sent in total.
 \newline

\end{proof}

\begin{remark}
After sampling $M$ type $2$ samples which costs $Mn$ we can first create a binary tree which costs $n \log_{2} n$. Afterwards the cost of sampling of one coordinate (type $1$ samples) is just $M \log_2 n$.
\end{remark}

\section{Additional Results}

This useful lemma is used to give an estimation of the area where we have to search for the dual solution:
\begin{lemma}
  \label{lm:upper bound of the Wasserstein dual solution}
  We can upper bound the norm of the dual solution by $
  \|\lambda^{*} \|_{2}\<= R  $ with
    $R^2  =  \dfrac{  2 n\suml{l=1}^{m} w_{l}^{2}   \|C_{l} \|^2_\infty }{   \lambda_{\min}^+ (W) } $
for $\gamma $ proportional to $\e$ and $c$ is a constant close to $2$ (see~\cite{lan2017communication}
and Lemma 8 in \cite{kroshnin2019complexity}).
\end{lemma}

We normalize the cost matrix $\| C \|_\infty^2 \<= 1$. For the weights $\omega_{i} = \dfr{1}{m} $ the estimate becomes
\begin{equation}
  \label{eq:estimate of R}
R \<= \sqrt{ \dfr{2n }{m  \lambda_{\min}^+ (W)} }.
\end{equation}

\section{Additional Simulation Results}

In this section, we show additional simulation results for the proposed Algorithm~\ref{Alg:mainWB}. In each case, we assume there is a network with $30$ nodes, and randomly generate $30$ different Gaussian distributions shown in Figure~\ref{fig:gaussians3}. Each node is assigned one of those Gaussian distributions, with uniformly sampled means and variances. Agents have the ability to query samples from those distributions, but are oblivious to the distribution itself. We test the proposed algorithm on five different graph classes: path, cycle, star, Erd\H{o}s-Renywe test three different setups for number of samples, namely, $M_1=1$ and $M_2=100$, $M_1=10$ and $M_2=10$, and $M_1=100$ and $M_2=1$. A video of the resulting barycenter estimates for the path graph can be found in \url{https://youtu.be/aSwNZvCkrCw}.

\begin{figure}[t!]
    \centering
    \includegraphics[trim=10 0 10 0,clip,width=0.9\linewidth]{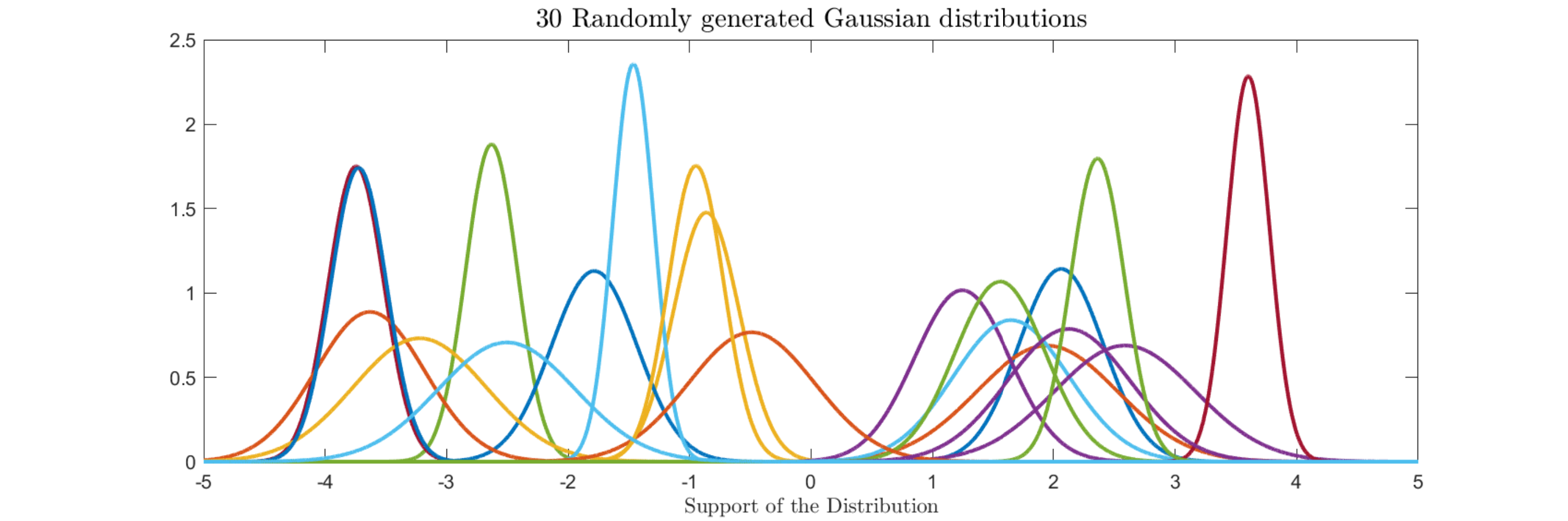}
    \caption{The randomly generated Gaussian distributions. Each distribution is assigned to each agent.}
    \vspace{-0.4cm}
    \label{fig:gaussians3}
\end{figure}

\begin{figure}[t!]
    \centering
    \includegraphics[trim=35 20 43 20,clip,width=0.8\linewidth]{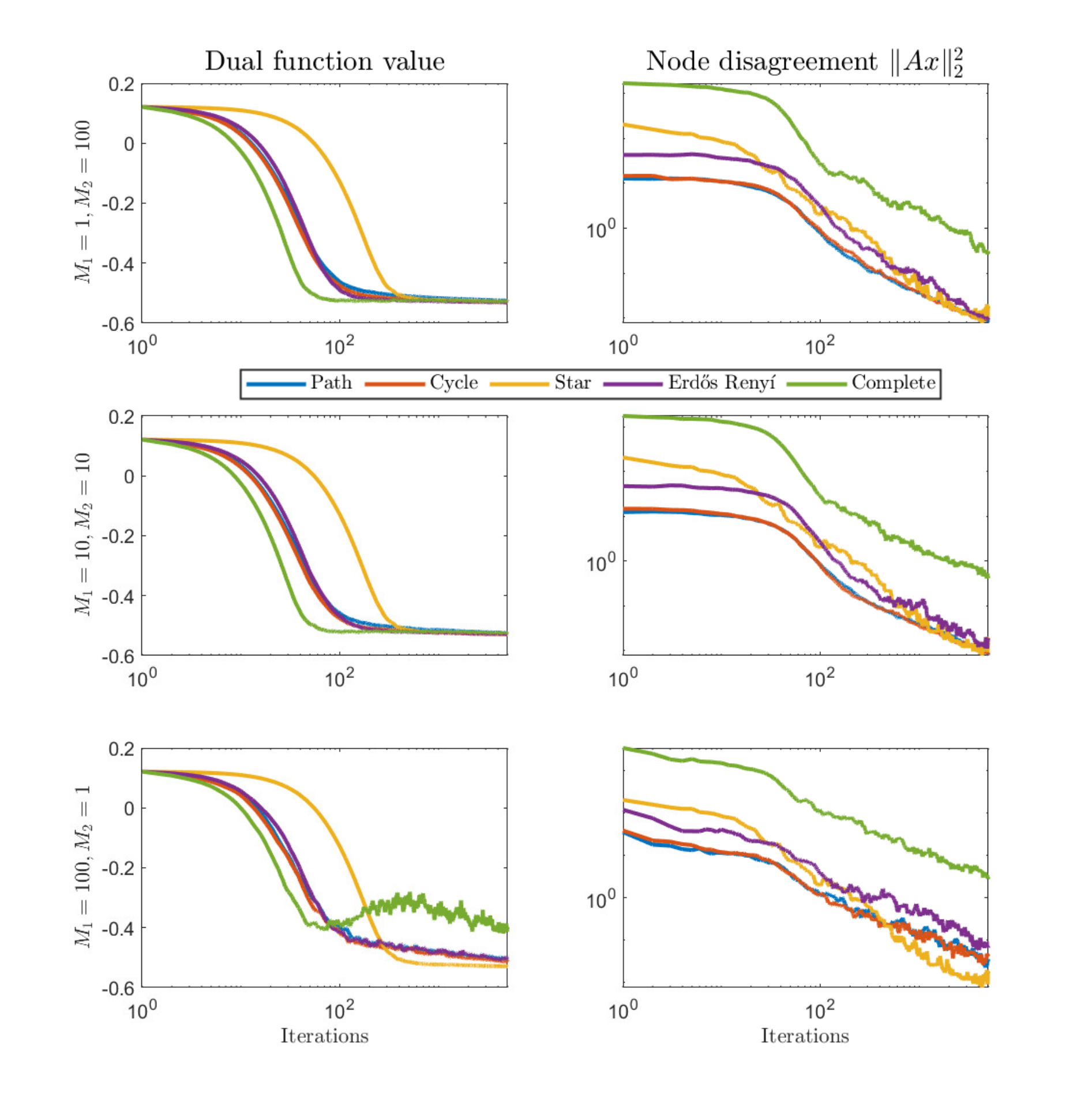}
    \caption{The dual function value and node disagreement of the iterates generated by  Algorithm~\ref{Alg:mainWB} on five different graphs, namely: path, cycle, star, Erd\H{o}s-Reny\'i, and complete graphs, with $30$ nodes each. Moreover, we show the effects of different number of samples per iterations. For each network, we test $M_1=1$ and $M_2=100$, $M_1=10$ and $M_2=10$, and $M_1=100$ and $M_2=1$.}
    \vspace{-0.4cm}
   \label{fig:additional}
\end{figure}

\end{document}